\documentclass[11pt]{book}
\usepackage{amsmath,amsthm,amsfonts,amssymb, amscd}
\usepackage[utf8]{inputenc}
\usepackage{hyperref}
\usepackage{tikz}
\usepackage{tikz}
\usepackage{hyperref}
\usepackage[utf8]{inputenc}
\usepackage{graphicx}
\usepackage{hyperref}
\usepackage[left=1in,right=1in,top=1in,bottom=1in]{geometry}
\newtheorem{theorem}{Theorem}[section]
\newtheorem{proposition}[theorem]{Proposition}

\newtheorem{lemma}[theorem]{Lemma}
\theoremstyle{definition}
\newtheorem{definition}[theorem]{Definition}
\newtheorem{example}[theorem]{Example}
\newtheorem{exercise}[theorem]{Exercise}
\newtheorem{remark}[theorem]{Remark}
\DeclareMathOperator\coker{coker}

\title{Applied Sheaf Theory For Multi-agent Artificial Intelligence (Reinforcement Learning) Systems: A Prospectus}
\author{Eric Schmid}
\date{April 2025}

\begin{document}

\maketitle

\tableofcontents

\chapter{Sheaf Theory and Sheaf Cohomology}
\section{Introduction: Local-to-Global Principles and Motivation}\label{sec:intro}
Many problems in mathematics exhibit a tension between \textbf{local} behavior and \textbf{global} behavior on a space. \emph{Sheaf theory} provides a systematic framework to track data locally on a topological space and determine how to ``glue'' or assemble this local data into global structures. This local-to-global perspective is pervasive in areas such as algebraic geometry, topology, differential geometry, and even logic and computer science. The guiding question is: \emph{when can local solutions or constructions be uniquely merged into a global one?} \cite{GallierQuaintance2024}

\subsection{The Local-to-Global Phenomenon.}
We illustrate the local vs.\ global distinction  with a few examples:
\begin{itemize}
    \item \textbf{Topology:} A continuous function can be defined locally on each set of an open cover of a space, but these local definitions might not agree globally. For instance, a manifold can be covered by coordinate patches (local coordinate functions exist), yet there may be no single global coordinate chart covering the entire manifold.
    \item \textbf{Algebraic Geometry:} Regular functions on an algebraic variety are defined locally on affine patches. Whether these local functions extend to a global regular function is a central question. In fact, one defines varieties by gluing together affine algebraic pieces, and not every local regular function extends globally.
    \item \textbf{Differential Geometry:} One can define vector fields or differential forms on local neighborhoods of a manifold, but there may be global obstructions. A classic example: a nowhere-vanishing local section of the tangent bundle exists in small patches, but on a non-orientable manifold it cannot be extended globally (this relates to the non-existence of a global nowhere-vanishing vector field, an obstruction measured by the orientability).
    \item \textbf{Gauge Theory (Physics):} Locally one can choose a gauge (a function describing a field, like the electromagnetic potential in each patch), but globally there might be a mismatch leading to nontrivial topological classes (e.g. a magnetic monopole yields a potential defined locally but not globally continuous, reflecting a non-zero second cohomology class).
\end{itemize}
In each case, the question is whether locally defined data (continuous functions, sections of a bundle, etc.) that agree on overlaps can be ``glued'' to a valid object on the union of the patches. Sheaf theory formalizes this situation by assigning data to every open set and imposing conditions that ensure consistency and gluability of that data. The advantage of the sheaf perspective is that it provides a language and tools (notably sheaf cohomology) to measure the \emph{obstructions} to gluing local data into global data.

\subsection{From Functions to Sheaves.}
A naive approach to a local-to-global problem is to consider the collection of all functions (or sections, etc.) defined on each open subset of a space. This leads to the notion of a \emph{presheaf}, which associates to each open $U$ some set (or algebraic structure) $F(U)$ of data on $U$, together with restriction maps to smaller opens. However, not every presheaf allows gluing of local sections; additional axioms are needed. A \emph{sheaf} is a presheaf that satisfies these axioms ensuring uniqueness and existence of glued sections. We will formalize these definitions in the next sections.

The power of sheaf theory comes not only from organizing data but also from sheaf cohomology, a machinery that measures the failure of the local-to-global principle. Roughly, if local data patch together perfectly, certain cohomology groups vanish; when there are obstructions, those appear as nonzero cohomology classes. Sheaf cohomology has become fundamental in modern algebraic geometry, providing powerful generalizations of classical results (such as the Riemann–Roch theorem, which was originally proved in the 19th century using purely classical methods) and many classification results. It establishes deep connections between topology, analysis, and algebraic geometry.

In this paper, we assume familiarity with basic category theory (categories, functors, limits/colimits) and basic algebraic topology (e.g. the notion of a chain complex and homology), but no prior knowledge of sheaf theory or homological algebra. We proceed from the definition of presheaves and sheaves, through examples and foundational constructions (morphisms, stalks, sheafification), and then focus on the theory of sheaf cohomology. Along the way, we include examples and exercises to illustrate key points. We use standard notation and conventions from algebraic topology and algebraic geometry texts.

\section{Presheaves}\label{sec}
In essence, a \textbf{presheaf} is a rule that assigns to each open set of a topological space some data (such as functions or other mathematical objects defined on that open set), with the requirement that these data are consistent under restriction to smaller opens. The elements of this data set are called \textbf{sections}, which intuitively can be thought of as functions or mathematical objects that are defined over a specific open set. Presheaves formalize the idea of ``local data structure on a space''. We give the formal definition and then illustrate with examples.

\begin{definition}[Presheaf]\label{def:presheaf}
Let $X$ be a topological space. A \emph{presheaf} $F$ on $X$ consists of the following data:
\begin{itemize}
    \item For each open set $U \subseteq X$, an object $F(U)$ in some category (such as sets, abelian groups, rings, or modules), whose elements are called \textbf{sections of $F$ over $U$}. We often denote a section $s \in F(U)$ by $s|_U$ when we wish to emphasize its domain.
    \item For each inclusion of open sets $V \subseteq U \subseteq X$, a \textbf{restriction map}
    \[
      \rho^F_{U,V}: F(U) \to F(V),
    \] 
    which is a morphism in the relevant category, often written as $s \mapsto s|_V$ for $s \in F(U)$, subject to the following axioms:
    \begin{itemize}
        \item \textbf{Identity:} For any open set $U$, the restriction to itself is the identity: if $id: U \hookrightarrow U$ is the identity inclusion, then $\rho^F_{U,U} = \mathrm{id}_{F(U)}$ (so every section restricts to itself on the same set).
        \item \textbf{Transitivity:} If $W \subseteq V \subseteq U$ are open sets, then restrictions compose as expected: $\rho^F_{V,W} \circ \rho^F_{U,V} = \rho^F_{U,W}$. In other words, restricting a section from $U$ to $W$ directly is the same as first restricting to $V$ then to $W$.
    \end{itemize}
\end{itemize}
The pair $(F,\rho^F)$ (often just $F$) is called a presheaf on $X$. Common examples include presheaves of abelian groups, rings, or modules, where each $F(U)$ has the corresponding algebraic structure and each restriction map preserves this structure. We often write $F: U \mapsto F(U)$, $U$ open in $X$. \cite{Rotman2009}
\end{definition}

Equivalently, for those familiar with category theory, a presheaf of sets on $X$ is the same as a contravariant functor from the category of open sets of $X$ (with inclusions as morphisms) to the category of sets. \cite{Rotman2009} In this language, $F$ assigns each open $U$ an object $F(U)$ and each inclusion $V \subseteq U$ a morphism $F(U) \to F(V)$ (the restriction), satisfying the functoriality conditions (identity and composition) automatically.

\begin{example}\label{ex:presheaf_examples}
We list some fundamental examples of presheaves on a topological space $X$:
\begin{enumerate}
    \item \textbf{Continuous Functions:} For a fixed target space (e.g. $\mathbb{R}$ or $\mathbb{C}$), let $F(U) = C^0(U)$ be the set of all continuous functions on $U$. For $V \subseteq U$, define $\rho_{U,V}(f) = f|_V$ by restriction of the domain. This clearly satisfies the identity and transitivity axioms, so $U \mapsto C^0(U)$ is a presheaf on $X$. Similarly, one can define presheaves $C^k(U)$ of $C^k$ (continuously $k$-times differentiable) functions on $U$, or $C^\infty(U)$ of smooth functions if $X$ has a differentiable manifold structure. \cite{Vakil2024}
    
    \item \textbf{Constant Presheaf:} Fix a set $S$. Define $F(U) = S$ for every open $U$, and let each restriction map $\rho_{U,V}: S \to S$ be the identity on $S$. This presheaf $F$ simply assigns the same set $S$ to every open set (intuitively pretending that a ``section'' over $U$ is just an element of $S$, with no further dependence on $U$). We call this the \emph{constant presheaf with fiber $S$}. It vacuously satisfies the presheaf axioms (identity and transitivity are trivial here). \emph{Warning:} This $F$ will not in general be a sheaf unless $X$ is connected (we will revisit this).
    
\item \textbf{Sections of a Bundle:} Let $\pi: E \to X$ be a (continuous) fiber bundle or more generally any map of spaces. Define a presheaf $F$ by $F(U) = \{\text{(continuous) sections } s: U \to E \text{ of }\pi\}$, i.e. maps $s$ such that $\pi \circ s = \mathrm{id}_U$. If $V \subseteq U$, restriction is given by $\rho_{U,V}(s) = s|_V$, the section obtained by restricting $s$'s domain. This is a presheaf: identity is clear, and if $W \subset V \subset U$, then $(s|_V)|_W = s|_W$. For instance, if $E$ is a topological vector bundle on $X$, $F(U)$ is the space of continuous (or smooth, etc.) sections of the bundle over $U$. In an important special case, $\pi$ could be the trivial bundle $X \times A \to X$ with fiber $A$ an abelian group; then sections on $U$ correspond to arbitrary continuous functions $U \to A$, and $F$ is the presheaf of $A$-valued functions on $X$. The special case of locally constant functions arises when $A$ has the discrete topology.
    
    \item \textbf{Open-set Sheaf:} Define $F(U) = \{ V \mid V \text{ is an open subset of } U\}$, the power set of $U$ restricted to opens. If $V \subseteq U$ are open, let $\rho_{U,V}: F(U) \to F(V)$ be given by intersection: for $W \in F(U)$ (so $W$ is open in $U$), $\rho_{U,V}(W) = W \cap V$, which is an open subset of $V$ and hence in $F(V)$. One checks easily that this satisfies the presheaf axioms. Moreover, it is actually a sheaf: if $\{U_i\}$ is an open cover of $U$ and we have a collection of open sets $W_i \in F(U_i)$ that agree on overlaps (i.e., $W_i \cap U_j = W_j \cap U_i$ for all $i,j$), then these glue uniquely to an open set $W = \bigcup_i W_i \in F(U)$ where $W \cap U_i = W_i$ for all $i$. This can be seen as a special case of example (1) where the target space is the Sierpinski space. Intuitively, this sheaf assigns to each open $U$ the collection of all its open subsets, providing a way to encode the topology of $X$ within the sheaf framework.
\end{enumerate}
\end{example}

Presheaves are very flexible, but the lack of additional conditions means they might not reflect the desired local-to-global properties. In particular, a presheaf does not require that sections which locally agree must be identical, nor that locally defined sections can always be glued into a global section. These crucial properties are imposed in the definition of a \emph{sheaf}, which we introduce next. The examples above illustrate typical presheaves. Before defining sheaves, we formalize how a presheaf can fail the ``local consistency'' requirements:

\begin{definition}[Morphisms of Presheaves]\label{def:presheaf_morphism}
If $F$ and $G$ are presheaves on $X$, a \emph{morphism of presheaves} $\phi: F \to G$ is a collection of maps $\{\phi_U: F(U) \to G(U)\}_{U \subseteq X,\,U \text{ open}}$ such that for every inclusion $V \subseteq U$, the following diagram commutes:
\[
\begin{CD}
F(U) @>\phi_U>> G(U)\\
@V\rho^F_{U,V}VV @VV\rho^G_{U,V}V\\
F(V) @>\phi_V>> G(V)\,,
\end{CD}
\] 
i.e. $\phi_V(s|_V) = (\phi_U(s))|_V$ for every $s \in F(U)$. This is just a natural transformation of functors in the categorical view. Morphisms of presheaves compose in the obvious way and one obtains the \emph{category of presheaves on $X$}. If each $\phi_U$ is bijective, we call $\phi$ an isomorphism of presheaves (and $F, G$ are isomorphic).
\end{definition}

Morphisms of presheaves allow us to compare different presheaf structures. In particular, any sheaf (to be defined) is in particular a presheaf, so these notions apply to sheaves as well. The category of sheaves on $X$ will be a full subcategory of the category of presheaves.

\section{Sheaves}\label{sec:sheaf}
A \textbf{sheaf} is a presheaf that satisfies the additional local consistency axioms: (1) if two sections agree locally, they must coincide globally (local identity, or separation axiom), and (2) any collection of locally compatible sections arises as the restriction of a unique global section (gluing, or existence axiom). These conditions ensure that the presheaf actually captures the idea of ``local data that determine a unique global outcome''. More formally:

\begin{definition}[Sheaf]\label{def:sheaf}
A presheaf $F$ on $X$ (of sets) is said to be a \emph{sheaf} if it satisfies the following \textbf{sheaf axioms} in addition to the presheaf axioms:
\begin{itemize}
    \item \textbf{Locality (Uniqueness/Separation):} If $\{U_i\}_{i \in I}$ is an open cover of an open set $U \subseteq X$, and if $s, t \in F(U)$ are two sections such that $s|_{U_i} = t|_{U_i}$ for \emph{all} $i\in I$, then $s = t$ in $F(U)$. In other words, a section is uniquely determined by its values on the members of an open cover.  Equivalently, if a section restricts to the zero (or identity) section on each piece of a cover, it must be zero; no nontrivial section can vanish on all pieces of a cover.
    \item \textbf{Gluing (Existence):} If $\{U_i\}_{i \in I}$ is an open cover of $U$ and we have a collection of sections $s_i \in F(U_i)$ for each $i$, which are \emph{locally compatible} in the sense that on every nonempty overlap $U_i \cap U_j$, the restrictions agree: $s_i|_{U_i \cap U_j} = s_j|_{U_i \cap U_j}$ for all $i,j$, then there exists at least one section $s \in F(U)$ (a \emph{global section} on $U$) such that $s|_{U_i} = s_i$ for all $i \in I$. This $s$ is said to \emph{glue} or \emph{extend} the collection $\{s_i\}$. \cite{Jagathese2023}
\end{itemize}
If $F$ is a sheaf, we often call its sections simply ``sections'' without always mentioning the open set (context usually clarifies the domain). A \textbf{sheaf of abelian groups (rings, etc.)} is defined similarly (the data are abelian groups, and restriction maps are homomorphisms), and the sheaf axioms above are imposed on the underlying sets of sections. 
\end{definition}

The sheaf axioms are often summarized by saying that for any open cover $\{U_i\}$ of $U$, the sequence 
\[ 0 \to F(U) \xrightarrow{\;\;\delta\;\;} \prod_i F(U_i) \rightrightarrows \prod_{i,j} F(U_i \cap U_j) \] 
is \emph{exact}. Let's explain this sequence in detail:

\begin{itemize}
\item The first map $F(U) \to \prod_i F(U_i)$ sends a global section $s \in F(U)$ to the collection of its restrictions $(s|_{U_i})_i$ to each open set in the cover.

\item The double arrow $\prod_i F(U_i) \rightrightarrows \prod_{i,j} F(U_i \cap U_j)$ represents two different restriction maps:
  \begin{itemize}
    \item The first restricts each $s_i \in F(U_i)$ to $s_i|_{U_i \cap U_j} \in F(U_i \cap U_j)$
    \item The second restricts each $s_j \in F(U_j)$ to $s_j|_{U_i \cap U_j} \in F(U_i \cap U_j)$
  \end{itemize}
\end{itemize}

``Exactness'' here means that the kernel of the second arrow equals the image of the first. In other words, a collection of sections $(s_i)_i \in \prod_i F(U_i)$ comes from a global section in $F(U)$ if and only if they agree on all overlaps: $s_i|_{U_i \cap U_j} = s_j|_{U_i \cap U_j}$ for all $i,j$.

This exactness encodes both sheaf axioms:
\begin{itemize}
\item \textbf{Locality axiom} (injectivity of $F(U) \to \prod_i F(U_i)$): If two global sections $s, t \in F(U)$ agree on each $U_i$ (i.e., $s|_{U_i} = t|_{U_i}$ for all $i$), then they must be identical ($s = t$). Equivalently, if $s \neq t$, there must exist some $U_i$ where $s|_{U_i} \neq t|_{U_i}$.

\item \textbf{Gluing axiom} (surjectivity onto the equalizer): If a collection of sections $s_i \in F(U_i)$ satisfies the compatibility condition $s_i|_{U_i \cap U_j} = s_j|_{U_i \cap U_j}$ for all $i,j$, then there exists a global section $s \in F(U)$ that restricts correctly to each $U_i$ (i.e., $s|_{U_i} = s_i$ for all $i$).
\end{itemize}

This categorical interpretation reveals how sheaves enforce the principle that compatible local data uniquely determine global data. In practice, we will mainly work with the concrete formulation given earlier.

A presheaf that satisfies only the Locality axiom is sometimes called a \emph{separated presheaf} (or \emph{pre-sheaf} in older literature), but the term and distinction won't be needed much here; we primarily care about full sheaves.

\begin{example}\label{ex:sheaf_examples}
Revisiting the presheaves of Example \ref{ex:presheaf_examples}, we identify which are sheaves:
\begin{enumerate}
    \item \textbf{Sheaf of Continuous Functions:} $U \mapsto C^0(U)$ is actually a sheaf. The locality axiom holds because if two continuous functions agree on each piece of an open cover, they agree everywhere on the union (since points of $U$ lie in some piece and the functions agree there). More formally, if $f, g \in C^0(U)$ with $f|_{U_i}=g|_{U_i}$ for all $i$, then for each $x\in U$ pick some $U_i$ containing $x$, and $f(x)=g(x)$ since $x\in U_i$. Thus $f=g$ globally. The gluing axiom also holds: given continuous $f_i \in C^0(U_i)$ that agree on overlaps $U_i\cap U_j$, one can construct a continuous function $f$ on $\bigcup_i U_i = U$ by defining $f(x) = f_i(x)$ if $x \in U_i$. This is well-defined because if $x\in U_i \cap U_j$, $f_i(x)=f_j(x)$ by hypothesis. To see $f$ is continuous, note that each $x$ has a neighborhood where $f$ equals one of the $f_i$, so $f$ is locally continuous and hence continuous on $U$. Moreover $f$ restricts to each $f_i$. Uniqueness is ensured by the locality axiom. Thus $C^0_X$, the presheaf of continuous real-valued functions on $X$, is a sheaf. By similar reasoning, $C^k_X$ and $C^\infty_X$ (on a differentiable manifold $X$) are sheaves of functions.
    
\item \textbf{Sheaf of Differentiable Functions:} As above, $F(U) = C^\infty(U)$ is a sheaf of smooth functions (on a smooth manifold $X$). The sheaf axioms hold because smoothness is a local property. For the gluing axiom, if we have locally smooth functions $f_i \in C^\infty(U_i)$ that agree on overlaps $U_i \cap U_j$, they produce a well-defined function $f$ on $U = \cup_i U_i$. This function $f$ is automatically smooth because smoothness is characterized by local behavior: a function is smooth if and only if it is smooth in a neighborhood of every point. Since every point $x \in U$ is contained in some $U_i$ where $f$ restricts to the smooth function $f_i$, the glued function $f$ inherits smoothness naturally. The locality axiom (uniqueness) is clear: if two smooth functions agree locally everywhere, they are identical.
    
    \item \textbf{Sections of a Vector Bundle:} For a continuous (or smooth) vector bundle $E \to X$, the presheaf $U \mapsto \{\text{sections of }E \text{ on }U\}$ is a sheaf. The sheaf axioms are usually part of the definition of a fiber bundle: local sections that agree on overlaps give a unique global section because locally (with respect to a trivializing cover) one can patch the sections using a partition of unity or directly by definition of bundles. Thus, the sheaf of sections of $E$ is often denoted $\Gamma(X,E)$ for global sections, and $\Gamma(U,E)$ for sections over $U$. If $E=X\times A$ is a trivial bundle with fiber an abelian group $A$, then $\Gamma(U,E)$ are just $A$-valued continuous functions on $U$. If $A$ is discrete (or $U$ is connected), these are precisely the locally constant functions. We denote by $\underline{A}_X$ the sheaf of locally constant $A$-valued functions on $X$, called the \emph{constant sheaf (with stalk $A$)}.
    
    \item \textbf{Sheaf of Holomorphic Functions:} In complex analysis or on a complex manifold $X$, the assignment $U \mapsto \mathcal{O}_X(U)$, where $\mathcal{O}_X(U)$ is the ring of holomorphic (complex-analytic) functions on $U$, defines a sheaf of rings on $X$. The gluing axiom in this case is a nontrivial theorem of complex analysis: holomorphic functions that agree on overlaps extend to a holomorphic function on the union (this follows from the identity theorem and analytic continuation). Thus $\mathcal{O}_X$ is a sheaf (in fact, a sheaf of rings, and even of $\mathbb{C}$-algebras) on $X$. This sheaf is fundamental in complex geometry; for example, its higher cohomology groups yield important invariants of $X$.
    \item \textbf{Constant Presheaf (revisited):} The presheaf $F$ with $F(U)=S$ for all $U$ (from Example \ref{ex:presheaf_examples}(2)) is \emph{not} a sheaf in general. The issue is the gluing axiom: suppose $X$ is the union of two disjoint nonempty open sets $U_1, U_2$ (so $X$ is disconnected). Take sections $s_1 \in S = F(U_1)$ and $s_2 \in S = F(U_2)$ that are different elements of $S$. On the overlap $U_1 \cap U_2 = \emptyset$, the compatibility condition is vacuously true (any condition on the empty set is trivially satisfied), so $\{s_1, s_2\}$ is a locally compatible family on the cover $\{U_1,U_2\}$. The gluing axiom would demand a section $s \in F(X)$ such that $s|_{U_i}=s_i$. But $F(X) = S$ only has one element if the sheaf property held (because on $X$ itself $s$ would equal $s_1$ on $U_1$ and $s_2$ on $U_2$, forcing $s_1 = s_2$ if $s$ existed, which is not the case by our choice). Thus no single element of $S$ can restrict to give two different values on $U_1$ and $U_2$. Therefore, the constant presheaf fails the gluing axiom (unless $S$ has only one element, or $X$ is connected so that any locally constant choice is globally constant). In fact, the \emph{sheafification} of this presheaf is exactly the sheaf of locally constant $S$-valued functions $\underline{S}_X$, whose sections on $U$ are those functions $f:U \to S$ which locally are constant (these can glue because on overlaps they agree automatically if defined to be locally constant). \cite{GallierQuaintance2024}
    
\item \textbf{Discontinuous Functions Presheaf:} The presheaf 

\[
P(U) = \{\text{all (not necessarily continuous) functions }U \to \mathbb{R}\}
\]

with ordinary restriction maps is actually a sheaf.

To see why, recall that for $P$ to be a sheaf, it must satisfy both the locality (uniqueness) and gluing axioms. Consider an open cover $\{U_i\}$ of some open set $U$. If we have sections $f_i \in P(U_i)$ that agree on all overlaps $U_i \cap U_j$, then we can define a function $f$ on $U$ by setting $f(x) = f_i(x)$ whenever $x \in U_i$. This is well-defined precisely because the sections agree on overlaps. This function $f$ is an element of $P(U)$ since $P$ allows all functions, not just continuous ones, and $f$ restricts to each $f_i$ on $U_i$. 

The locality axiom is also satisfied: if $f, g \in P(U)$ restrict to the same function on each $U_i$ of an open cover, then $f$ and $g$ must be identical on $U$, as they agree at each point.

This example relates to example (1) in that it can be viewed as the sheaf of continuous functions where $\mathbb{R}$ is given the indiscrete topology, which makes every function continuous by definition.

\end{enumerate}
\end{example}

Intuitively, a sheaf formalizes the principle that \emph{compatible local data determine a unique global object}. 

Just as we defined morphisms of presheaves, we have:

\begin{definition}[Morphisms of Sheaves]\label{def:sheaf_morphism}
A \emph{morphism of sheaves} $\varphi: \mathcal{F} \to \mathcal{G}$ on $X$ is simply a morphism of the underlying presheaves (Definition \ref{def:presheaf_morphism}). In other words, a family of maps $\{\varphi_U: \mathcal{F}(U) \to \mathcal{G}(U)\}_{U\subseteq X}$ commuting with all restriction maps. Sheaves on $X$ and their morphisms form the \textbf{category of sheaves} on $X$, denoted $\mathbf{Sh}(X)$. 

A morphism $\varphi: \mathcal{F}\to\mathcal{G}$ is an \emph{isomorphism} (or sheaf isomorphism) if each $\varphi_U$ is bijective. In that case $\mathcal{F}$ and $\mathcal{G}$ are essentially the same sheaf.
\end{definition}

\begin{example}
Basic examples of sheaf morphisms include inclusion maps or restriction of structure: e.g. the inclusion of the sheaf of differentiable functions into the sheaf of continuous functions $C^\infty_X \hookrightarrow C^0_X$ (each smooth function is in particular continuous) is a morphism of sheaves. Another example is differentiation: on a smooth manifold $X$, the assignment $d: C^\infty_X(U) \to \Omega^1_X(U)$ given by $f \mapsto df$ (the exterior derivative) defines a morphism of sheaves from the sheaf of smooth functions to the sheaf of 1-forms, since $d(f|_V) = (df)|_V$ for all $V\subseteq U$ (differentiation commutes with restriction).
\end{example}

\begin{definition}[Exact Sequences of Sheaves]\label{def:exact_sequence}
Given a sequence of sheaf morphisms $0 \to \mathcal{F}' \xrightarrow{\;\alpha\;} \mathcal{F} \xrightarrow{\;\beta\;} \mathcal{G} \to 0$, we say it is \emph{exact} if:
\begin{itemize}
\item $\alpha$ is a monomorphism of sheaves (i.e., $\alpha_U$ is injective for all open sets $U$)
\item $\beta$ is an epimorphism of sheaves (which is a weaker condition than being surjective on all sets of sections; it only requires $\beta$ to be locally surjective)
\item The image of $\alpha$ equals the kernel of $\beta$ (i.e., $\mathcal{F}'$ is identified as the kernel subsheaf of $\beta$)
\end{itemize}
\end{definition}

Exact sequences of sheaves are fundamental when we define cohomology, since applying the global section functor $\Gamma(X,-)$ to an exact sequence of sheaves generally produces a \emph{left exact} sequence of groups, and the right-end failure of exactness leads to cohomology groups (see \S\ref{sec:cohomology_intro}). We note here that the category of sheaves of abelian groups on $X$ is an \textbf{abelian category}: it has a zero object (the sheaf $\mathbf{0}$ with $\mathbf{0}(U)=\{0\}$ for all $U$), all kernels and cokernels exist, and all monomorphisms and epimorphisms are normal. In fact, one can show that if $\mathcal{F}$ and $\mathcal{G}$ are sheaves of abelian groups, then $\ker(\beta)$ defined as the presheaf $U \mapsto \ker(\beta_U: \mathcal{F}(U) \to \mathcal{G}(U))$ is actually a sheaf (since the sheaf axioms can be checked to hold for kernels, being a sub-presheaf of $\mathcal{F}$). However, $\mathrm{coker}(\alpha)$ defined by $U \mapsto \coker(\alpha_U)$ is generally not a sheaf but only a presheaf; to get the cokernel in the category of sheaves, one must sheafify this presheaf. Thus $\mathbf{Sh}(X, \mathbf{Ab})$ (the category of sheaves of abelian groups on $X$) is abelian. Moreover, $\mathbf{Sh}(X, \mathbf{Ab})$ has \emph{enough injectives} (we will discuss this in \S\ref{sec:derived_functors}), meaning any sheaf can be embedded in an injective sheaf. These facts allow us to do homological algebra with sheaves.

\section{Stalks and the Étale Space}\label{sec:stalks}
A key feature of sheaves is their ability to encode both local and global information seamlessly. The notion of a \textbf{stalk} of a sheaf formalizes the idea of ``germs of sections'' at a point -- capturing the infinitesimal or very local behavior of sections. By examining stalks, we can often reduce questions about sheaves to simpler questions about these local germs.

\begin{definition}[Stalk and Germ]\label{def:stalk}
Let $\mathcal{F}$ be a presheaf (or sheaf) on $X$, and let $x \in X$ be a point. The \emph{stalk} of $\mathcal{F}$ at $x$, denoted $\mathcal{F}_x$, is defined as the direct limit (or colimit) of the sets $F(U)$ as $U$ ranges over all open neighborhoods of $x$:
\[
\mathcal{F}_x := \varinjlim_{U \ni x} \mathcal{F}(U).
\]
Concretely, an element of $\mathcal{F}_x$ is an equivalence class of pairs $(U, s)$ where $U$ is an open neighborhood of $x$ and $s \in \mathcal{F}(U)$ is a section defined on $U$. \cite{Vakil2024} Two pairs $(U,s)$ and $(V,t)$ are considered equivalent if there exists an open neighborhood $W \subseteq U \cap V$ of $x$ such that $s|_W = t|_W$ in $\mathcal{F}(W)$. An equivalence class of such pairs is called a \emph{germ} of a section at $x$, often denoted by $[s]_x$ or simply $s_x$. We write $\mathrm{germ}_x(s) \in \mathcal{F}_x$ for the germ of $s$ at $x$. There are natural \emph{projection maps} $\rho_{U,x}: \mathcal{F}(U) \to \mathcal{F}_x$ sending $s \mapsto [s]_x$. By construction, if $x \in V \subseteq U$, then $\rho_{V,x}(s|_V) = \rho_{U,x}(s)$, ensuring consistency.
\end{definition}

The stalk $\mathcal{F}_x$ can be thought of as ``the set of all values that sections of $\mathcal{F}$ take at $x$, up to identifying two sections that agree in some neighborhood of $x$''. If $\mathcal{F}$ is a sheaf of, say, abelian groups or rings, then each stalk $\mathcal{F}_x$ naturally inherits the structure of an abelian group or ring (since direct limits preserve such structures): operations are defined on representatives and checked to be well-defined. 

\begin{theorem}[Exactness on Stalks]\label{thm:exactness_stalks}
A sequence of sheaves is exact if and only if it induces an exact sequence on every stalk. That is, a sequence
\[ 0 \to \mathcal{F}' \xrightarrow{\;\alpha\;} \mathcal{F} \xrightarrow{\;\beta\;} \mathcal{G} \to 0 \]
is exact if and only if for every point $x \in X$, the induced sequence on stalks
\[ 0 \to \mathcal{F}'_x \xrightarrow{\;\alpha_x\;} \mathcal{F}_x \xrightarrow{\;\beta_x\;} \mathcal{G}_x \to 0 \]
is exact.
\end{theorem}

\begin{proof}
This follows from the fact that many sheaf properties are \emph{local} and can be checked on stalks. For injectivity, a morphism of sheaves $\varphi: \mathcal{F} \to \mathcal{G}$ is injective (monomorphic) if and only if for every point $x \in X$, the induced map on stalks $\varphi_x: \mathcal{F}_x \to \mathcal{G}_x$ is injective. Similarly, surjectivity of $\varphi$ can be checked stalkwise (this requires that every section of $\mathcal{G}$ has a local preimage in $\mathcal{F}$ around each point). 

For exactness in the middle (i.e., $\text{im}(\alpha) = \ker(\beta)$), we note that taking stalks commutes with taking kernels and images, so $(\ker \beta)_x = \ker(\beta_x)$ and $(\text{im} \alpha)_x = \text{im}(\alpha_x)$. Therefore, $\text{im}(\alpha) = \ker(\beta)$ as sheaves if and only if $\text{im}(\alpha_x) = \ker(\beta_x)$ for all stalks. 

Thus, exactness of the sequence of sheaves is equivalent to exactness at each stalk.
\end{proof}

This principle often simplifies verification of exactness, as stalk calculations are typically more straightforward than checking exactness for all open sets.

The importance of stalks is that many properties of sheaves can be checked by looking at stalks. For instance, as mentioned, a sheaf morphism $\varphi: \mathcal{F}\to\mathcal{G}$ is injective if and only if $\varphi_x: \mathcal{F}_x \to \mathcal{G}_x$ is injective for all $x$. Similarly, surjectivity and exactness can be checked on stalks (this is because sheaf conditions localize statements, and direct limits are exact functors in the category of abelian groups, etc.). In other words, stalks allow one to reduce global questions to local ones around each point.

\begin{example}
For the sheaf $\mathcal{O}_X$ of holomorphic functions on a complex manifold $X$, the stalk $\mathcal{O}_{X,x}$ is the ring of all germs of holomorphic functions at the point $x$. This is exactly the \emph{local ring} of $X$ at $x$, which in algebraic geometry or complex analysis is a fundamental object (e.g., $\mathcal{O}_{X,x}$ might be $\mathbb{C}\{z\}$, the ring of convergent power series, if $X=\mathbb{C}$ and $x=0$). For the constant sheaf $\underline{A}_X$ (locally constant functions with value in $A$), the stalk $\underline{A}_{X,x}$ is just $A$ itself for every $x$, since any germ of a locally constant function is determined by its constant value on some neighborhood of $x$.
\end{example}

The process of passing to stalks is sometimes called \emph{localization} in this context, by analogy with localizing algebraic objects. The slogan is: \emph{Sheaves are determined by their stalks, and maps of sheaves are determined by their effect on stalks.}

\begin{proposition}[Sheaf Equality via Stalks]\label{prop:stalk_equal}
If $\mathcal{F}$ is a sheaf and $s, t \in \mathcal{F}(U)$ are such that $s_x = t_x$ in $\mathcal{F}_x$ for all $x \in U$ (meaning the germs of $s$ and $t$ at every point of $U$ agree), then $s = t$ in $\mathcal{F}(U)$. In other words, a sheaf section is uniquely determined by all its germs.
\end{proposition}
\begin{proof}
For each $x\in U$, $s_x = t_x$ implies there is a neighborhood $V_x \subseteq U$ of $x$ such that $s|_{V_x} = t|_{V_x}$. The collection $\{V_x\}_{x\in U}$ forms an open cover of $U$. By the sheaf locality axiom, since $s$ and $t$ agree on each $V_x$ in the cover, they must be equal on $U$. 
\end{proof}

This shows the intuition that a sheaf is a kind of ``local object''---knowledge of it on arbitrarily small neighborhoods (stalks) reconstructs it globally.

Sheaves also have a beautiful \emph{geometric} interpretation via the notion of an \textbf{étalé space} (sometimes spelled ``étale space''). This is another way to visualize a sheaf as a topological space mapping down to $X$:

\begin{definition}[Étale Space of a Sheaf]\label{def:etale_space}
Let $\mathcal{F}$ be a sheaf on $X$. The \emph{étalé space} of $\mathcal{F}$, denoted $E(\mathcal{F})$ or sometimes just $E$, is the disjoint union of all stalks of $\mathcal{F}$:
\[ 
E(\mathcal{F}) = \bigsqcup_{x \in X} \mathcal{F}_x\,,
\] 
equipped with a natural topology and a projection map $\pi: E(\mathcal{F}) \to X$ defined by sending each element of $\mathcal{F}_x$ to the point $x$ (so $\pi|_{\mathcal{F}_x}: \mathcal{F}_x \to \{x\}$ is trivial). The topology on $E(\mathcal{F})$ is defined by specifying a basis of open sets: for each open $U \subseteq X$ and each section $s \in \mathcal{F}(U)$, define 
\[ 
\widetilde{s}(U) := \{\, s_y \in \mathcal{F}_y \mid y \in U \,\} \subseteq E(\mathcal{F})\,,
\] 
the set of germs of $s$ at all points of $U$. One checks that $\pi(\widetilde{s}(U)) = U$. The collection of all such $\widetilde{s}(U)$, for all $U$ and all $s \in \mathcal{F}(U)$, forms a basis for a topology on $E(\mathcal{F})$ under which $\pi: E(\mathcal{F}) \to X$ becomes a local homeomorphism (étale map). $E(\mathcal{F})$ is called the étalé (or \emph{sheaf}) space of $\mathcal{F}$.
\end{definition}

The topology on $E(\mathcal{F})$ is carefully chosen to establish a precise correspondence between sections of the sheaf $\mathcal{F}$ and continuous sections of the projection map $\pi$. Specifically, the basic open sets $\widetilde{s}(U)$ ensure that sections of $\mathcal{F}$ induce continuous maps $\sigma: U \to E(\mathcal{F})$ defined by $\sigma(x) = s_x$. 

The condition that $\pi: E(\mathcal{F}) \to X$ is a local homeomorphism means every point in $E(\mathcal{F})$ has a neighborhood that is mapped homeomorphically onto an open set of $X$. In fact, the sets $\widetilde{s}(U)$ themselves provide such neighborhoods: each $\widetilde{s}(U) \subseteq E(\mathcal{F})$ is homeomorphic (via $\pi$) to $U$. This local homeomorphism property ensures that the local nature of a section $s$ guarantees that its corresponding map $\sigma$ is continuous. This is why this construction is called ``étale'' (French for roughly \emph{spread out flat}): the sheaf is realized as a covering-like space over $X$.

The sheaf axioms have natural interpretations in this topology:
\begin{itemize}
\item The locality axiom ensures that sections are uniquely determined by their germs, which corresponds to section maps $\sigma: U \to E(\mathcal{F})$ being uniquely determined by their values.
\item The gluing axiom ensures that locally compatible sections can be glued together, which corresponds to local section maps that agree on overlaps being able to be glued into a continuous global section map.
\end{itemize}

Intuitively, the étalé space is like taking each possible germ as an actual point lying over the original base space. For example, for the sheaf $\mathcal{C}^0_X$ of continuous real functions, it's important to note that $E(\mathcal{C}^0_X)$ is not simply $X \times \mathbb{R}$ as one might naively think. Although the space of sections of $X \times \mathbb{R}$ is the sheaf of continuous functions, the projection from $X \times \mathbb{R}$ to $X$ is not a local homeomorphism, which is a required property of étalé spaces. The correct étalé space for continuous functions has a more complex structure that does satisfy the local homeomorphism condition. If we instead took the sheaf of locally constant $\mathbb{R}$-valued functions, its étalé space would be $X \times \mathbb{R}$ with the discrete topology on $\mathbb{R}$, which is a covering space of $X$. Thus covering spaces correspond to sheaves that are locally constant. In general, an étalé space can be seen as a sort of ``variable set'' varying continuously over $X$.

Every section $s \in \mathcal{F}(U)$ corresponds to a continuous map (actually a section in the topological sense) $\sigma: U \to E(\mathcal{F})$ of the étalé space, defined by $\sigma(x) = s_x$ (the germ of $s$ at $x$). The condition that $s$ is a sheaf section exactly ensures that $\sigma$ is continuous and $\pi \circ \sigma = \mathrm{id}_U$. Thus there is a natural bijection:

\[
\mathcal{F}(U) \;\cong\; \{\text{continuous sections } \sigma: U \to E(\mathcal{F}) \text{ of }\pi\}\,,
\]

valid for every open $U \subseteq X$. In other words, giving a sheaf $\mathcal{F}$ is equivalent to giving a local homeomorphism $\pi: E \to X$ (where $E$ is the étalé space) together with this correspondence of sections. In category-theoretic terms, the functor $\mathcal{F} \mapsto E(\mathcal{F})$ establishes an equivalence between the category of sheaves on $X$ and the category of étalé spaces over $X$ (local homeomorphisms to $X$). We will not prove this formally here, but it is a beautiful perspective.

\begin{remark}
The étalé space perspective allows one to think of a sheaf as a kind of ``spread-out function or object''. For instance, one can visualize the sheaf of germs of a function as literally taking each point of $X$ and attaching all possible germ values at that point as a fiber. If $\mathcal{F}$ is a sheaf of sets, $E(\mathcal{F})$ is like a topological bundle (not necessarily with a single fiber type). If $\mathcal{F}$ is a sheaf of groups or rings, each stalk $\mathcal{F}_x$ has that algebraic structure, and one can often turn $E(\mathcal{F})$ into a topological group (but caution: generally it’s only a space fibered in groups, not a group globally unless something like a trivialization exists).
\end{remark}

The étalé space construction also provides an alternative way to construct the \emph{sheafification} of a presheaf, as we will see next.

\section{Sheafification}\label{sec:sheafification}
Not every presheaf is a sheaf (as we saw). However, an important fact is that for any presheaf, there is a best possible approximation of it by a sheaf, called its \textbf{sheafification}. Sheafification is analogous to, say, group completion or abelianization: it’s a functor that makes a presheaf into a sheaf in a universal way. 

\begin{theorem}[Existence of Sheafification]\label{thm:sheafification_exists}
For any presheaf $P$ on $X$, there exists a sheaf $P^+$ on $X$ together with a morphism of presheaves $\eta: P \to P^+$ satisfying the following universal property: for any sheaf $\mathcal{F}$ on $X$ and any presheaf morphism $\phi: P \to \mathcal{F}$, there exists a unique sheaf morphism $\phi^+: P^+ \to \mathcal{F}$ such that $\phi = \phi^+ \circ \eta$. \cite{GallierQuaintance2024} In categorical terms, $P^+$ is the sheafification of $P$, and the functor $P \mapsto P^+$ is left adjoint to the inclusion functor from sheaves to presheaves. Moreover, $P^+$ is unique up to isomorphism by this universal property.
\end{theorem}

In simpler terms, $P^+$ is a sheaf containing $P$ as a sub-presheaf (via $\eta$), and any attempt to map $P$ into a sheaf factors uniquely through $P^+$. It's important to note that the natural map $\eta: P(U) \to P^+(U)$ is not necessarily injective for all $U$. Nevertheless, $P^+$ is the ``smallest'' sheaf that contains $P$ in a suitable sense. If $P$ is already a sheaf, then $P^+ \cong P$ (sheafification doesn't change a sheaf).

There are several ways to explicitly construct $P^+$. We describe a common construction via stalks and germs:

Let $P$ be a presheaf. Define a candidate for sections of $P^+$ over an open set $U$ as:

\[
P^+(U) := \{\, s: U \to \bigsqcup_{x \in U} P_x \mid s(x) \in P_x \text{ for each $x$, and $s$ is ``locally a germ of a single section of $P$''} \,\}\,.
\]

In other words, an element of $P^+(U)$ is like a section of the étalé space of $P$ over $U$, where now $P$ need not be a sheaf so we consider all germs in each stalk $P_x$. Formally, one can define $P^+(U)$ as the set of functions $s$ assigning to each $x \in U$ a germ $s_x \in P_x$ such that for each $x \in U$, there exists a neighborhood $V \subseteq U$ of $x$ and some section $t \in P(V)$ whose germ at every point $y \in V$ equals $s_y$. In other words, $s$ is locally (around every point) coming from an actual presheaf section. Such an $s$ is essentially a section of the étalé space of $P$ that is continuous in the sense of the basis given by germs. We define restriction of such an $s$ to a smaller open $W \subseteq U$ in the obvious way: $(s|_W)(x) = s(x)$ for $x \in W$. This indeed defines a sheaf $P^+$ (one must check the sheaf axioms, which follow from the construction). By construction, there's a natural presheaf morphism $\eta: P \to P^+$ sending each section $t \in P(U)$ to the section of $P^+$ given by $x \mapsto$ germ of $t$ at $x$. The fact that $P^+$ satisfies the universal property of sheafification is a bit technical but essentially stems from the fact that any map out of $P$ into a sheaf factors through maps on germs.

Another (equivalent) perspective: The above construction is reminiscent of first forming the étalé space $E(P) = \bigsqcup_{x} P_x$ (which is a set of germs) and then taking sections in the topological sense. Essentially, $P^+$ is the sheaf of \emph{continuous sections of the étalé space of $P$}. This is sometimes called the \emph{étalé sheafification}. There is also a more algebraic description: one can first force the identity axiom (separation) by quotienting $P$ by the relation identifying sections that agree on an open cover, and then force the gluing axiom by an additional colimit construction. But the stalk/germ method is more intuitive.

\begin{example}
Consider again the constant presheaf $P$ with $P(U)=S$ for all $U$. Its stalk at any $x$ is 
\[P_x = \varinjlim_{U \ni x} S = S,\] 
since any germ is represented by an element of $S$ on some neighborhood. The sheafification $P^+$ is the sheaf of locally constant functions with values in $S$. Indeed, an element of $P^+(U)$ is by definition a choice of a germ in $S$ at each $x \in U$ which locally comes from a single presheaf section. But presheaf sections of $P$ are just constant functions on $U$ (since $P(V)=S$ for any $V$). So the condition that $s(x)$ is locally a germ of a single section means: for each $x$ there is a neighborhood $V$ of $x$ and an element $a \in S$ such that for all $y \in V$, $s(y)$ is the germ of the constant section $a$. But the germ of the constant section $a$ at $y$ is just $a$ itself. So $s(y) = a$ for all $y$ in $V$. Thus $s$ is a locally constant function $U \to S$. Hence $P^+(U)$ can be identified with $\{f: U \to S \mid f \text{ is locally constant}\}$. This is exactly the usual constant sheaf $\underline{S}_X$ (the sheaf of locally constant $S$-valued functions). Therefore $\underline{S}_X$ is the sheafification of the constant presheaf $P$. We saw earlier that $P$ was not a sheaf unless $S$ is trivial or $X$ connected, but $\underline{S}_X$ is a genuine sheaf. The map $\eta: P(U)=S \to P^+(U)$ sends the single value $a \in S$ to the constant function $f(y)=a$ on $U$. This clearly satisfies the universal property: any sheaf receiving a map from $P$ (i.e. giving an element of $S$ for each open in a compatible way) must factor through $\underline{S}_X$ (which picks out the locally constant function determined by those chosen values).
\end{example}

\begin{remark}[Functoriality of Sheafification]
The sheafification functor $(\cdot)^+$ is functorial: a presheaf morphism $\phi: P \to Q$ induces a sheaf morphism $\phi^+: P^+ \to Q^+$. Functoriality follows naturally by defining $\phi^+$ pointwise: given a germ $[s]_x \in (P^+)_x$, we set  
\[
\phi^+([s]_x) = [\phi(s)]_x,
\]
ensuring compatibility with restrictions. This functoriality means that sheafification is a \emph{functor} left adjoint to the inclusion of sheaves into presheaves. Sheafification also commutes with taking stalks: $(P^+)_x \cong P_x$ for each $x$. This is intuitive since the stalk is already a very local object, the sheaf axioms don't change the germs at a single point.
\end{remark}

Having established the basic theory of sheaves (definitions, examples, morphisms, stalks, sheafification), we have the machinery needed to study \textbf{sheaf cohomology}. We have hinted that sheaf cohomology arises from the failure of the sequence 
\[ 0 \to F(U) \to \prod_i F(U_i) \rightrightarrows \prod_{i,j} F(U_i \cap U_j) \to \cdots \] 
to be exact beyond the first term. In fact, continuing this sequence leads to the \emph{\v{C}ech cohomology} of a sheaf. We will now introduce cohomology both from the elementary \v{C}ech viewpoint and the more general derived-functor viewpoint, then relate them and discuss key properties and examples.

\section{Sheaf Cohomology}\label{sec:cohomology_intro}
One of the main motivations for introducing sheaves is that they allow the definition of \textbf{sheaf cohomology}, which measures the extent to which the process of forming global sections fails to be exact. Cohomology provides obstructions to gluing local data and is an invariant that connects algebraic topology with algebraic geometry and analysis.

\subsection{Motivation: Cohomology as an Obstruction.}
Consider a sheaf $\mathcal{F}$ of abelian groups on $X$ (the abelian condition is needed to have a cohomology theory in the usual sense). We have the exact sequence for any open cover $\{U_i\}$ of an open $U$:
\[ 0 \to \mathcal{F}(U) \to \prod_i \mathcal{F}(U_i) \to \prod_{i<j} \mathcal{F}(U_i \cap U_j) \to \cdots \,.\]
The sheaf axioms tell us that the first arrow is injective (identity axiom) and the kernel of the second arrow equals the image of the first (sections that agree on overlaps come from a global section; gluing axiom). However, the second arrow $\prod_i \mathcal{F}(U_i) \to \prod_{i<j} \mathcal{F}(U_i \cap U_j)$ need not be surjective in general; if it isn't, that means there is a collection of sections $\{t_{ij} \in \mathcal{F}(U_i \cap U_j)\}$ that are pairwise compatible on triple overlaps, but which do not come as differences of a bunch of sections $\{s_i \in \mathcal{F}(U_i)\}$. In other words, there is an \emph{obstruction} to finding local sections $s_i$ that realize given overlap data $t_{ij}$. This obstruction is measured by the \textbf{\v{C}ech cohomology} group $\check{H}^1(\{U_i\}, \mathcal{F})$ for that cover. More generally, the continued failure of exactness at higher stages leads to higher cohomology groups $\check{H}^p(\{U_i\}, \mathcal{F})$.

Thus, cohomology arises as soon as we have a nontrivial cycle of data that is not a boundary of something from a previous stage. For sheaves, $\check{H}^0$ is just global sections, and $\check{H}^1$ detects the obstruction to gluing sections (this often corresponds to torsors or twisting of objects; e.g. line bundles are classified by $H^1$ of the sheaf of invertible functions). Higher $H^p$ can similarly be interpreted (though such interpretations get more abstract). 

One can also motivate sheaf cohomology via more classical cohomology theories. For example, singular cohomology of a space $X$ can be recovered as sheaf cohomology of the constant sheaf $\underline{\mathbb{Z}}_X$ (under mild conditions on $X$). Sheaf cohomology provides a unified framework that encompasses many classical cohomology theories. In particular, the de Rham cohomology on a smooth manifold $M$ is isomorphic to the sheaf cohomology of the de Rham complex $\Omega^\bullet$:
\[ H_{\text{dR}}^*(M) \cong H^*(M, \Omega^\bullet). \]
Sheaf cohomology is extremely general: it works for any sheaf of abelian groups on any topological space, and in algebraic geometry one uses it for the Zariski or étale topology, etc., where other tools are not readily available.

\emph{Sheaf cohomology measures the failure of local data to determine global data exactly}. Next, we define the two main approaches to sheaf cohomology: \v{C}ech cohomology and derived functor cohomology.

\subsection{\v{C}ech Cohomology.}\label{sec:cech_cohom}
Let $\mathcal{F}$ be a sheaf of abelian groups on $X$. Choose an open cover $\mathfrak{U} = \{U_i\}_{i\in I}$ of $X$. The \textbf{\v{C}ech complex} of $\mathcal{F}$ with respect to this cover is the cochain complex:
\[ 
C^p(\mathfrak{U}, \mathcal{F}) := \prod_{i_0,i_1,\dots,i_p \in I} \mathcal{F}(U_{i_0} \cap U_{i_1} \cap \cdots \cap U_{i_p})\,,
\] 
the product of sections on all $(p+1)$-fold intersections of the cover sets. An element of $C^p(\mathfrak{U},\mathcal{F})$ can be thought of as an assignment of a section on each $p$-fold intersection $U_{i_0 \dots i_p} := U_{i_0}\cap \cdots \cap U_{i_p}$. In low degrees: 
\begin{itemize}
    \item $C^0(\mathfrak{U},\mathcal{F}) = \prod_{i} \mathcal{F}(U_i)$, a 0-cochain is a choice of section $s_i \in \mathcal{F}(U_i)$ for each $i$ (an open set in the cover).
    \item $C^1(\mathfrak{U},\mathcal{F}) = \prod_{i,j} \mathcal{F}(U_i \cap U_j)$, a 1-cochain is a choice of section $t_{ij} \in \mathcal{F}(U_i \cap U_j)$ for each ordered pair $(i,j)$.
    \item etc.
\end{itemize}
We define the \textbf{\v{C}ech differential} $\delta: C^p \to C^{p+1}$ by an alternating sum that is analogous to the simplicial cohomology differential:
For a $p$-cochain $c = \{ c_{i_0\dots i_p} \}$ (where $c_{i_0\dots i_p} \in \mathcal{F}(U_{i_0\dots i_p})$), 
\[ (\delta c)_{i_0 i_1 \dots i_{p+1}} = \sum_{k=0}^{p+1} (-1)^k\, c_{i_0 \dots \widehat{i_k} \dots i_{p+1}} \big|_{U_{i_0 \dots i_{p+1}}}\,. \] 
Here the hat indicates omission of that index, and the restriction $|_{U_{i_0 \dots i_{p+1}}}$ means we restrict the section from a $p$-fold intersection down to the $(p+1)$-fold intersection that is contained in it. This formula is the usual alternating sum but with the important aspect that we must restrict each term to the common intersection $U_{i_0 \dots i_{p+1}}$ in order to subtract sections living on the same domain. 

One checks that $\delta^{2}=0$ (this relies on the presheaf condition for $\mathcal{F}$, essentially) so $(C^\bullet(\mathfrak{U},\mathcal{F}), \delta)$ is a cochain complex. Here, $C^\bullet(\mathfrak{U},\mathcal{F})$ denotes the collection of all $C^p(\mathfrak{U},\mathcal{F})$ for all $p \geq 0$, where the superscript bullet ($\bullet$) is a standard notation indicating the entire complex of cochain groups in all degrees. The cohomology of this complex, 
\[ 
\check{H}^p(\mathfrak{U}, \mathcal{F}) := H^p(C^\bullet(\mathfrak{U},\mathcal{F})) = \ker(\delta: C^p \to C^{p+1}) / \mathrm{im}(\delta: C^{p-1} \to C^p)\,,
\] 
is called the \textbf{\v{C}ech cohomology of $\mathcal{F}$ with respect to the cover $\mathfrak{U}$}. It is a priori a cover-dependent notion. For example, $\check{H}^0(\mathfrak{U},\mathcal{F})$ is 
\[ \ker(\delta: C^0 \to C^1) = \{ (s_i) \in \prod_i \mathcal{F}(U_i) \mid s_i|_{U_i \cap U_j} = s_j|_{U_i \cap U_j}\ \forall i,j \}\,,\] 
which by the sheaf gluing axiom is isomorphic to $\mathcal{F}(X)$ (the global sections). So $\check{H}^0(\mathfrak{U},\mathcal{F}) \cong \mathcal{F}(X)$ for any cover $\mathfrak{U}$. Meanwhile, a 1-cocycle is a collection $(t_{ij})$ with $t_{ij} \in \mathcal{F}(U_i \cap U_j)$ such that $\delta(t)_{ijk}=0$, which explicitly means 
\[ t_{jk} - t_{ik} + t_{ij} = 0 \quad \text{in }\mathcal{F}(U_{ijk}) \] 
for all $i,j,k$. This is precisely the condition that $\{t_{ij}\}$ is a compatible family on triple overlaps. Such a $\{t_{ij}\}$ is called a 1-cocycle. It is a coboundary (1-coboundary) if $t_{ij} = s_i|_{U_{ij}} - s_j|_{U_{ij}}$ for some sections $s_i \in \mathcal{F}(U_i)$. Thus $\check{H}^1(\mathfrak{U},\mathcal{F})$ consists of equivalence classes of 1-cocycles under those trivialized by 0-cochains. This matches the description of gluing obstructions earlier.

If we refine the cover (take a cover $\mathfrak{V}$ that is a refinement of $\mathfrak{U}$), there are natural restriction maps 
\[ \check{H}^p(\mathfrak{U},\mathcal{F}) \to \check{H}^p(\mathfrak{V},\mathcal{F}) \] 
that induce an inverse system. One then defines the \textbf{\v{C}ech cohomology of $\mathcal{F}$ on $X$} (with no cover specified) as the direct limit over all open covers:
\[ 
\check{H}^p(X,\mathcal{F}) := \varinjlim_{\mathfrak{U}} \check{H}^p(\mathfrak{U},\mathcal{F})\,.
\] 

\cite{Rotman2009}

In many nice situations, a single good cover yields the same cohomology as the direct limit (for example, if $X$ is a paracompact space, any open cover has a refinement which is \emph{uniformly fine} enough that further refinement doesn't change the cohomology). In the context of manifolds or CW complexes, one often works with a good cover (e.g. all intersections are contractible) which simplifies computations. 

We note a couple of important facts:
\begin{itemize}
    \item For a sheaf of abelian groups, $\check{H}^p(X,\mathcal{F})$ is functorial in $\mathcal{F}$: a sheaf morphism $\mathcal{F} \to \mathcal{G}$ induces maps on cochains and hence on cohomology. So cohomology is a functor $H^p: \mathbf{Sh}(X,\mathbf{Ab}) \to \mathbf{Ab}$ (from sheaves to abelian groups).
\item If $X$ is a reasonably nice space (e.g. paracompact Hausdorff), then \v{C}ech cohomology for sheaves is isomorphic to the sheaf cohomology defined via derived functors (discussed next), which is the more general definition. For paracompact Hausdorff spaces, this isomorphism is typically proven via a spectral sequence argument, often using the Godement resolution, which relates the two cohomology theories. However, for pathological spaces or sheaves, \v{C}ech cohomology might differ from derived functor cohomology (Grothendieck famously gave conditions when they agree and when they might not). In algebraic geometry, one often can use \v{C}ech cohomology computed with affine open covers since higher cohomology vanishes on affines.
    \item \v{C}ech cohomology can be used to derive long exact sequences: given a short exact sequence of sheaves $0 \to \mathcal{F}' \to \mathcal{F} \to \mathcal{F}'' \to 0$, one can construct a long exact sequence in \v{C}ech cohomology $\cdots \to \check{H}^p(X,\mathcal{F}') \to \check{H}^p(X,\mathcal{F}) \to \check{H}^p(X,\mathcal{F}'') \to \check{H}^{p+1}(X,\mathcal{F}') \to \cdots$ (the connecting homomorphism comes from a cocycle construction). Under conditions where \v{C}ech agrees with derived functor cohomology, this is the same long exact sequence we get from derived functors.
\end{itemize}

\subsection{Derived Functor Definition.}\label{sec:derived_functors}
A more abstract (but powerful) approach to sheaf cohomology is via \textbf{derived functors}. The key observation is that the global section functor $\Gamma(X,-): \mathbf{Sh}(X,\mathbf{Ab}) \to \mathbf{Ab}$ (which takes a sheaf $\mathcal{F}$ to $\mathcal{F}(X)$) is a left-exact functor between abelian categories. Indeed, $0 \to \mathcal{F}' \to \mathcal{F} \to \mathcal{F}'' \to 0$ exact implies $0 \to \Gamma(X,\mathcal{F}') \to \Gamma(X,\mathcal{F}) \to \Gamma(X,\mathcal{F}'')$ is exact, since taking global sections is just evaluating at an open (and exactness at that first stage is precisely sheaf locality for $X$ itself, which holds). However, $\Gamma(X,-)$ is not in general right-exact; the image of $\Gamma(X,\mathcal{F}) \to \Gamma(X,\mathcal{F}'')$ might not equal $\Gamma(X,\mathcal{F}'')$ if not every local section of $\mathcal{F}''$ lifts globally in $\mathcal{F}$. That failure is exactly measured by higher derived functors.

Because $\mathbf{Sh}(X,\mathbf{Ab})$ is an abelian category with enough injectives, we can take an injective resolution of any sheaf. If $\mathcal{F}$ is a sheaf, choose an injective resolution:
\[ 
0 \to \mathcal{F} \to I^0 \to I^1 \to I^2 \to \cdots 
\] 
with each $I^p$ injective (such resolutions exist because sheaves have enough injectives). Now apply the functor $\Gamma(X,-)$ to this complex. $\Gamma(X,I^p)$ is an abelian group, and since $I^p$ is injective, $\Gamma(X,I^p)$ is an exact functor of $\mathcal{F}$ (so higher cohomology of an injective is zero). We obtain a cochain complex:
\[ 0 \to \Gamma(X,I^0) \to \Gamma(X,I^1) \to \Gamma(X,I^2) \to \cdots, \] 
whose cohomology groups are by definition the \textbf{sheaf cohomology groups}:
\[ 
H^p(X,\mathcal{F}) := H^p(\Gamma(X,I^\bullet))\,. 
\] 

\cite{GallierQuaintance2024}

These are the \emph{right derived functors} $R^p \Gamma(X,\mathcal{F})$. By general homological algebra, this definition is independent of the choice of injective resolution and satisfies the usual properties: for instance, it gives a long exact cohomology sequence for any short exact sequence of sheaves (since one can splice resolutions or apply the snake lemma to the double complex formed by two resolutions, etc.). Moreover, one shows that $H^0(X,\mathcal{F}) \cong \Gamma(X,\mathcal{F})$ (because $\Gamma(X,-)$ is left exact, so $H^0$ yields its value), and for an injective sheaf $I$, $H^p(X,I)=0$ for $p>0$.

\[ H^p(X,\mathcal{F}) = R^p \Gamma(X,\mathcal{F}) \]
is the $p$th right derived functor of global sections. This definition is elegant and conceptual, and it coincides with \v{C}ech cohomology for sufficiently nice spaces (in fact, always there is a canonical map $\check{H}^p(X,\mathcal{F}) \to H^p(X,\mathcal{F})$ that is an isomorphism if $\mathcal{F}$ is what is known as a \emph{flasque} or \emph{soft} or other acyclic sheaf or if $X$ is paracompact, etc.).

A concrete way to compute sheaf cohomology using derived functors is to use a resolution by acyclic sheaves (ones whose higher cohomology vanishes). For example, \textbf{flasque sheaves} are those $\mathcal{L}$ such that $\mathcal{L}(X) \to \mathcal{L}(U)$ is surjective for all open $U$; they are soft and satisfy $H^p(X,\mathcal{L})=0$ for all $p>0$ (since one can always extend local sections to global sections, which implies any \v{C}ech cocycle is a coboundary). Every sheaf has a flasque resolution, so one can compute cohomology as the cohomology of the global section complex of a flasque resolution (this is another approach to showing existence of $H^p$). This often simplifies computations in practice compared to injective resolutions (which are abstract).

The derived functor viewpoint also immediately gives functoriality in the sheaf argument, long exact sequences, and powerful tools like spectral sequences (Leray spectral sequence relates $H^p(X, R^q f_* \mathcal{F})$ to $H^{p+q}(Y,\mathcal{F})$ for a continuous map $f: Y \to X$) and base change theorems. It is the standard approach in advanced texts like Hartshorne's \emph{Algebraic Geometry} \cite{Hartshorne1977} and Grothendieck's work \cite{Grothendieck1957}.

$H^p(X,\mathcal{F})$ is a sequence of abelian groups (or modules) associated to each sheaf $\mathcal{F}$. They are zero for $p<0$, and $H^0(X,\mathcal{F}) = \mathcal{F}(X)$. They are functorial in $\mathcal{F}$ and fit into long exact sequences. They also often coincide with classical cohomology theories in special cases (e.g., if $\mathcal{F}=\underline{A}_X$ for a constant sheaf $A$ on a reasonably nice space, then $H^p(X,A) \cong H^p_{\text{sing}}(X;A)$, the singular cohomology of $X$ with coefficients in $A$).

\subsection{\v{C}ech vs. Derived Functor Cohomology.}
As noted, \v{C}ech cohomology $\check{H}^p(X,\mathcal{F})$ is not obviously the same as $H^p(X,\mathcal{F})$ in general, but there are conditions under which they agree. Typically, if $X$ is paracompact (or has a basis that is good for covers) and $\mathcal{F}$ is a sheaf of abelian groups (or more specifically a fine sheaf, etc.), then $\check{H}^p \cong H^p$. More concretely, if $\mathcal{F}$ is such that every open cover's higher \v{C}ech cohomology eventually stabilizes (which is true for paracompact spaces), then one can show any injective resolution gives the same result, and that result matches the direct limit of \v{C}ech. 

In practice, one often uses \v{C}ech cohomology to compute sheaf cohomology because it's more combinatorial. For example, in complex geometry, to compute $H^p(X,\mathcal{O}_X)$ (cohomology of the structure sheaf) one often uses a cover by affines and computes \v{C}ech cohomology since each affine piece has trivial cohomology and sections on intersections are computable, etc.

Grothendieck showed that if $\mathcal{F}$ is a \textbf{fine sheaf} on a paracompact space, then $H^p(X,\mathcal{F}) = 0$ for all $p>0$. \cite{GallierQuaintance2024} A sheaf $\mathcal{F}$ is called \textbf{fine} if for any locally finite open cover $\{U_i\}_{i \in I}$ of $X$, there exists a family of sheaf endomorphisms $\{\phi_i: \mathcal{F} \to \mathcal{F}\}_{i \in I}$ such that:
\begin{enumerate}
    \item For each $i \in I$, the support of $\phi_i$ is contained in $U_i$ (meaning $\phi_i(s)|_{X \setminus U_i} = 0$ for all sections $s$)
    \item $\sum_{i \in I} \phi_i = \text{id}_{\mathcal{F}}$ (the identity morphism on $\mathcal{F}$)
\end{enumerate}

This definition generalizes the notion of partitions of unity to the setting of sheaves. For example, the sheaf of smooth functions or differential forms on a manifold is fine via the existence of smooth partitions of unity. Fine sheaves are a subset of flasque sheaves (in fact fine $\implies$ flasque), so they are acyclic. Therefore, e.g. the sheaf of smooth functions $C^\infty_X$ on a paracompact manifold has no higher cohomology ($H^p(X,C^\infty_X)=0$ for $p>0$). Similarly, the sheaf of continuous real functions is fine (if paracompact Hausdorff), so it has no cohomology above 0. This doesn't mean the space has no topology; rather, it means that these sheaves are too flexible to capture interesting invariants (in contrast, constant sheaves are not fine, and indeed $H^p(X,\underline{\mathbb{R}})$ recovers real cohomology of $X$ which can be nonzero).

One important situation of agreement is for \textbf{coherent analytic sheaves} on complex manifolds or coherent algebraic sheaves on varieties: in those contexts, one often uses \v{C}ech cohomology with respect to an open cover by contractible (or affine) sets to compute $H^p$. For example, on a complex manifold, Dolbeault's theorem states $H^q(X,\Omega^p_X)$ (sheaf cohomology of holomorphic $p$-forms) is isomorphic to the $(p,q)$th Dolbeault cohomology of the manifold, which is computed using fine resolutions (the Dolbeault complex of $C^\infty$ forms). 

$\check{H}^p(X,\mathcal{F})$ and $H^p(X,\mathcal{F})$ coincide in most situations of interest, and one usually denotes them simply as $H^p(X,\mathcal{F})$. We will assume henceforth that we are working in a context where this is true (e.g. $X$ is paracompact).

\subsection{Properties of Sheaf Cohomology.}
We list some key properties and theorems of sheaf cohomology (mostly consequences of it being a derived functor):
\begin{itemize}
    \item \textbf{Long Exact Sequence:} If $0 \to \mathcal{F}' \to \mathcal{F} \to \mathcal{G} \to 0$ is an exact sequence of sheaves (of abelian groups), there is a natural long exact sequence in cohomology:
    \[
    0 \to H^0(X,\mathcal{F}') \to H^0(X,\mathcal{F}) \to H^0(X,\mathcal{G}) \xrightarrow{\;\delta\;} H^1(X,\mathcal{F}') \to H^1(X,\mathcal{F}) \to H^1(X,\mathcal{G}) \to \cdots\,,
    \] 
    continuing as $H^1 \to H^2$ and so on. The connecting homomorphism $\delta: H^0(X,\mathcal{G}) \to H^1(X,\mathcal{F}')$ measures the obstruction to lifting a global section of $\mathcal{G}$ to $\mathcal{F}$: given $s \in \mathcal{G}(X)$, one can locally lift it (by surjectivity on stalks of $\mathcal{F} \to \mathcal{G}$) to sections of $\mathcal{F}$, and the difference of two local lifts defines a \v{C}ech 1-cocycle in $\mathcal{F}'$ whose class is $\delta(s)$. This general mechanism is analogous to how connecting maps in algebraic cohomology measure extension classes.
    
    \item \textbf{Functoriality:} If $f: Y \to X$ is a continuous map and $\mathcal{F}$ a sheaf on $X$, there are two important functors between sheaf categories:

\begin{enumerate}
    \item \textbf{Pullback (inverse image)}: The pullback sheaf $f^{-1}\mathcal{F}$ on $Y$ is defined as the sheafification of the presheaf that assigns to each open set $V \subset Y$ the direct limit:
    \[
    (f^{-1}\mathcal{F})^{\text{pre}}(V) = \varinjlim_{U \supset f(V)} \mathcal{F}(U)
    \]
    where the limit is taken over all open sets $U \subset X$ containing $f(V)$. Intuitively, $f^{-1}\mathcal{F}$ is the sheaf on $Y$ that "pulls back" the sections of $\mathcal{F}$ via $f$.
    
    \item \textbf{Pushforward (direct image)}: The pushforward sheaf $f_*\mathcal{G}$ on $X$ for a sheaf $\mathcal{G}$ on $Y$ is defined by:
    \[
    (f_*\mathcal{G})(U) = \mathcal{G}(f^{-1}(U))
    \]
    for any open set $U \subset X$. In other words, sections of $f_*\mathcal{G}$ over $U$ are precisely the sections of $\mathcal{G}$ over the preimage $f^{-1}(U)$.
\end{enumerate}

These functors induce maps on cohomology (contravariant for pullback, covariant for pushforward). In particular, if $f: Y \to X$ is a continuous map of nice spaces, there is the \textbf{Leray spectral sequence}:
\[
E_2^{p,q} = H^p\!\Big(X, R^q f_* \mathcal{G}\Big) \implies H^{p+q}(Y,\mathcal{G})\,,
\] 
which relates cohomology of $Y$ to that of $X$ with higher direct images of $\mathcal{G}$. As a special case, if $f$ is a covering map (or any map such that higher direct images vanish), then $H^n(Y,\mathcal{G}) \cong H^n(X, f_*\mathcal{G})$. For example, if $Y$ is a covering of $X$ with fiber $F$ and $\mathcal{G}$ is a locally constant sheaf on $Y$, then $f_*\mathcal{G}$ is a locally constant sheaf on $X$ with fiber $\mathcal{G}(F)$, and one recovers that $H^*(Y,\mathcal{G}) \cong H^*(X, f_*\mathcal{G})$ (this is a sheafified version of the statement that cohomology of the total space of a covering is cohomology of base with local coefficients).
    
    Another consequence is \textbf{Base Change}: under certain conditions, for a fiber square with $f$ proper and some conditions on $\mathcal{F}$:
\begin{align*}
\begin{array}{ccc}
Y' & \stackrel{g}{\longrightarrow} & Y \\
\downarrow^{f'} & & \downarrow^{f} \\
X' & \stackrel{h}{\longrightarrow} & X
\end{array}
\end{align*}
we have $h^* R^q f_* \mathcal{F} \cong R^q f'_* g^* \mathcal{F}$. This is important in algebraic geometry (cohomology and base change theorems).
    
    \item \textbf{Vanishing Theorems:} Many conditions ensure vanishing of cohomology above a certain degree. For example, if $X$ has covering dimension $d$ (roughly, it can be covered by a finite refinement where no point is in more than $d+1$ sets), then $H^p(X,\mathcal{F}) = 0$ for all $p > d$ for any sheaf $\mathcal{F}$. This is because one can find a cover with nerve of dimension $d$. Specifically, a topological $d$-sphere $S^d$ has $H^p(S^d,\underline{\mathbb{Z}})=0$ for $p>d$.
    
    \item \textbf{Relation to Classical Cohomology:} If $X$ is a reasonably nice topological space (say a CW complex) and $A$ is an abelian group, then $H^p(X,\underline{A}_X) \cong H^p_{\text{sing}}(X;A)$, the usual singular (or Čech) cohomology with coefficients in $A$. For example, $H^1(X,\underline{\mathbb{Z}})$ classifies covering spaces of $X$ (as $\text{Hom}(H_1(X),\mathbb{Z})$ by the Hurewicz isomorphism, or equivalently $\text{Hom}(\pi_1(X),\mathbb{Z})$ if $X$ is nice). Also $H^1(X,\mathcal{O}_X^*)$ for a complex manifold $X$ is (by the exponential sheaf sequence) isomorphic to $H^2(X,\underline{\mathbb{Z}})$, giving the first Chern class of line bundles (this is a deep fact bridging analytic and topological cohomology via the exponential exact sequence of sheaves $0 \to 2\pi i \mathbb{Z} \to \mathcal{O}_X \xrightarrow{\exp} \mathcal{O}_X^* \to 0$).
\end{itemize}

\subsection{Examples and Applications.}
We conclude with several illustrative examples.

\begin{example}[Sheaf Cohomology of a Circle]\label{ex:circle_cohomology}
Let $X = S^1$ be the circle, and consider the constant sheaf $\underline{\mathbb{R}}_X$ (locally constant real-valued functions). We compute $H^p(S^1,\underline{\mathbb{R}})$ using a simple open cover of $S^1$ by two arcs $U_1, U_2$ that overlap in two disconnected segments. Using \v{C}ech cohomology: $\check{H}^0$ gives $\mathbb{R}$ (global constant functions). A 1-cocycle is $(t_{12})$ with $t_{12} \in \mathcal{F}(U_1 \cap U_2)$ such that $\delta t =0$ (here $\delta t$ on triple overlaps is trivial since we have only two sets). So a 1-cocycle is just a section on $U_1 \cap U_2$ (which has two connected components corresponding to the two overlaps). The condition to be a coboundary is that $t_{12} = s_1|_{12} - s_2|_{12}$ for some $s_i$ on $U_i$. But since $U_1 \cap U_2$ has two components, a globally constant function cannot produce an arbitrary pair of values on those two overlaps unless they are equal. Thus one finds $\check{H}^1(S^1,\underline{\mathbb{R}}) \cong \mathbb{R}$, corresponding to assigning a difference $t_{12}$ of value $a$ on one overlap and $-a$ on the other (coming from a sort of discontinuous global section if you try to glue one way around vs the other). In fact, $H^1(S^1,\underline{\mathbb{R}}) \cong \mathrm{Hom}(\pi_1(S^1),\mathbb{R}) \cong \mathbb{R}$, which matches singular cohomology $H^1(S^1,\mathbb{R}) \cong \mathbb{R}$. All higher $H^p=0$ for $p>1$ because $S^1$ has topological dimension 1.
\end{example}

Sheaf cohomology is thus a unifying language: singular cohomology, de Rham cohomology, etc., all fit into the framework by choosing appropriate sheaves (constant sheaf, differential forms sheaf, etc.). It also provides new invariants like the cohomology of structure sheaves in algebraic geometry (leading to definitions of irregularity, geometric genus, etc.). In algebraic topology, one rarely explicitly speaks of sheaf cohomology except to use local coefficient systems; however, in modern geometry and number theory, sheaf cohomology is indispensable.

\begin{exercise}
Compute $H^p(X,\underline{\mathbb{Z}})$ for $X = S^n$ (the $n$-sphere) for various $p$. You can use a good cover of $S^n$ by two contractible open sets whose intersection is homotopy equivalent to $S^{n-1}$. (This should recover the known singular cohomology of spheres.)
\end{exercise}

\begin{exercise}
Show that if $\mathcal{F}$ is a flasque sheaf on $X$, then $H^p(X,\mathcal{F}) = 0$ for all $p > 0$. (\textit{Hint:} For any open cover, every \v{C}ech $p$-cocycle is a coboundary because you can successively extend sections. Alternately, use the fact that $\mathcal{F}$ surjects onto any section on an open set to build a contracting homotopy for the \v{C}ech complex.)
\end{exercise}

\begin{exercise}
Consider the short exact sequence of sheaves on a manifold $X$: 
\[0 \to \underline{\mathbb{Z}}_X \to \underline{\mathbb{R}}_X \to \underline{\mathbb{R}}/\mathbb{Z}_X \to 0,\] 
where $\underline{\mathbb{R}}/\mathbb{Z}_X$ is the locally constant sheaf with fiber $\mathbb{R}/\mathbb{Z} \cong S^1$. Show that this induces the long exact sequence in cohomology that is isomorphic to the integral cohomology Bockstein sequence:
\[ \cdots \to H^p(X,\underline{\mathbb{R}}/\mathbb{Z}) \xrightarrow{\;\delta\;} H^{p+1}(X,\underline{\mathbb{Z}}) \to H^{p+1}(X,\underline{\mathbb{R}}) \to \cdots.\] 
Conclude that $H^p(X,\underline{\mathbb{R}}/\mathbb{Z})$ is isomorphic to the torsion subgroup of $H^{p+1}(X,\mathbb{Z})$ (primary decomposition of integral cohomology).
\end{exercise}

Sheaf cohomology provides a powerful and general way to handle cohomological questions across topology, geometry, and algebra. It formalizes the passage from local to global and measures the obstructions encountered. The theory is rich with algebraic tools (spectral sequences, etc.) and geometric interpretations (via examples like line bundles, divisor class groups, etc.).

\chapter{Prospectus}

\section{Introduction and Background}

\subsection{Motivation}
Whether analyzing financial markets, social networks, or multi-agent AI systems, our tools for capturing interactions between agents, information flow, and emergent properties remain frustratingly limited. The gap is particularly evident when trying to maintain computational efficiency while preserving theoretical rigor. My research tackles this problem by turning to sheaf theory – a branch of mathematics explicitly designed to relate local and global properties.

The insight driving this research comes from a surprising observation: many problems in economics and reinforcement learning share a fundamental characteristic. Local decisions made by individual agents with limited information collectively determine global system behavior. This mirrors precisely the mathematical domain where sheaf theory excels – analyzing how local data and constraints generate global structures. While economists, computer scientists, and mathematicians have all circled these problems, they've largely done so independently, missing opportunities for cross-pollination.

Sheaf theory has already demonstrated its power in mathematics and physics, but its application to computational contexts remains surprisingly underdeveloped. The time is ripe to bridge this gap, particularly for multi-agent and economic systems where its local-to-global perspective offers a natural fit. I'm proposing to develop a formalized sheaf-theoretic approach to modeling complex systems of interacting agents, with explicit applications to economic markets and reinforcement learning frameworks.

\subsection{Related Work}
My research builds on several pioneering efforts spanning mathematics, computer science, and economics:

\begin{itemize}
    \item Sterling's Agda implementation of sheaves has broken ground in constructive sheaf semantics, but hasn't been extended to multi-agent systems
    \item Escardo's type-topology library in univalent foundations provides essential tools for homotopy type theory that can inform our approach
    \item Recent work in compositional game theory has revealed the potential of category-theoretic methods in economics, though without leveraging sheaf cohomology
    \item Advances in languages like Agda and Idris have made formal verification of complex mathematical structures increasingly feasible
\end{itemize}

My preliminary research suggests a striking gap – despite sheaf cohomology's natural fit for multi-agent systems, direct applications to economics or reinforcement learning remain scarce. This presents both a challenge and an opportunity to develop novel frameworks with significant theoretical and practical impact.

\section{Research Questions and Objectives}

My research confronts four questions:

\begin{enumerate}
    \item What novel insights emerge when we examine multi-agent systems through the lens of sheaf cohomology? Can we characterize emergent properties that traditional approaches miss?
    
    \item How can we effectively represent economic and reinforcement learning systems as sheaves? What's the optimal mapping between system components and sheaf-theoretic constructs?
    
    \item How can we formalize these sheaf-theoretic representations in dependently typed languages to create verified computational foundations?
    
    \item How do we translate these formalizations into computationally efficient implementations that scale to real-world applications?
\end{enumerate}

To address these questions, I'll pursue five core objectives:

\begin{enumerate}
    \item Develop a cohomological interpretation of system properties that reveals new insights about emergent behaviors in multi-agent systems
    
    \item Establish formal mappings between multi-agent system components and sheaf-theoretic constructs, with specific instantiations for economic modeling and reinforcement learning
    
    \item Create a comprehensive formalization of the relevant sheaf structures in a dependently typed language
    
    \item Implement efficient computational methods for working with these structures, exploring parallel evaluation using interaction combinators
    
    \item Demonstrate practical applications through concrete case studies in economics and reinforcement learning
\end{enumerate}

\section{Methodology and Research Plan}

I'll tackle this ambitious agenda through five phases:

\subsection{Phase 1: Survey and Analysis of Existing Formalizations (6 months)}
I'll begin with a deep dive into:
\begin{itemize}
    \item Formalizations in category theory, sheaf theory, and topology
    \item Sheaf neural networks (Hansen and Gebhart)
    \item Persistent homology and Topological Data Analysis (Carlsson)
    \item Information theory through group cohomology (Vigneaux)
    \item Applications of sheaf theory to multi-agent systems
    \item Implementations in dependently typed languages
    \item Parallel evaluation architectures like Taelin's HVM2
\end{itemize}

I'll critically evaluate how these disparate approaches might complement or challenge each other. I'll start by immersing myself in Hansen and Gebhart's work on sheaf neural networks, Carlsson's pioneering contributions to persistent homology, and Vigneaux's information-theoretic models using group cohomology. Beyond surveying existing work, I'll assess implementations like Sterling's sheaf formalization and Escardo's type-topology library to identify leverageable components. I'm particularly interested in exploring how Taelin's interaction combinator approach might revolutionize parallel evaluation of sheaf cohomology computations.

\subsection{Phase 2: Identification of Target Applications (6 months)}
Armed with theoretical foundations from Phase 1, I'll identify high-impact applications:
\begin{itemize}
    \item Map economic modeling and reinforcement learning challenges that could benefit from sheaf-theoretic approaches
    \item Identify problems where local-to-global properties are crucial (market interactions, information aggregation, etc.)
    \item Develop preliminary sheaf-theoretic formulations
    \item Establish evaluation criteria for successful formulations
\end{itemize}

This phase will yield concrete target problems in economics and reinforcement learning that will drive my subsequent formalization work. I'll focus on applications where sheaf representations offer genuine insights or computational advantages that traditional approaches miss.

\subsection{Phase 3: Sheaf-Theoretic Modeling Framework (12 months)}
Building on my work in Phases 1 and 2, I'll develop a comprehensive framework for modeling multi-agent systems using sheaf theory:

\begin{itemize}
    \item Map system components (agents, information, decisions) to sheaf-theoretic constructs
    \item Define mechanisms by which global properties emerge from local behaviors through sheaf cohomology
    \item Formalize information flows and decision-making as sheaf morphisms
    \item Develop specific models for the target applications identified earlier
\end{itemize}

This phase will yield a formal mathematical framework establishing the "isomorphism" between multi-agent systems and sheaf-theoretic constructs, with concrete instantiations for my target applications.

\subsection{Phase 4: Formalization in Dependently Typed Languages (6 months)}
I'll implement the mathematical framework from Phase 3 in a dependently typed language:

\begin{itemize}
    \item Use the Curry-Howard isomorphism to encode sheaves and related categorical constructs in Agda or Idris
    \item Verify critical mathematical properties through formal proofs
    \item Implement instances of the framework for target applications
    \item Establish foundations for computational experiments
\end{itemize}

\subsection{Phase 5: Computational Implementation (6 months)}
I'll translate theoretical formalization into efficient implementations, exploring several approaches:

\begin{enumerate}
    \item \textbf{Sheaf cohomology implementation}: Algorithms for computing sheaf cohomology over multi-agent systems
    
    \item \textbf{Persistent homology / TDA approach}: Topological data analysis methods as alternatives or complements to sheaf cohomology
    
    \item \textbf{Information-theoretic implementations}: Approaches based on Vigneaux's group cohomology framework for Shannon entropy
    
    \item \textbf{HVM compilation}: Pathways to interaction combinators for efficient parallel evaluation via Taelin's HVM2 architecture
    
    \item \textbf{Hybrid approach}: Agda formalization with optimized Python implementations
\end{enumerate}

This phase involves comparative evaluation of these approaches, assessing computational efficiency, theoretical fidelity, and applicability to my target problems. While I'll emphasize the HVM approach for its parallel computation benefits, I'll also explore how persistent homology and information-theoretic implementations might offer unique advantages for specific domains.

\subsection{Phase 6: Application and Evaluation (6 months)}
I'll apply my framework to concrete problems in:

\begin{itemize}
    \item Economic modeling, potentially focusing on micro-economy integration or market actor connections
    \item Reinforcement learning systems, emphasizing incentives, policy formulation, and metrics
    \item Evaluation of effectiveness, computational efficiency, and theoretical insights
\end{itemize}

\section{Expected Contributions}

My research aims to make five significant contributions:

\begin{enumerate}

    \item Novel theoretical mappings between multi-agent systems and sheaf-theoretic constructs – providing new ways to understand complex systems
    
    \item A verified formalization of sheaves and sheaf cohomology in dependently typed languages – bridging pure mathematics and computational implementation
    
    \item Efficient computational implementations suitable for practical applications – moving beyond theoretical models to usable tools
    
    \item Demonstrated applications in economics and reinforcement learning – showing real-world relevance
    
    \item A foundation for future work on category-theoretic approaches to complex systems – opening new research directions
\end{enumerate}

The core scientific contribution will be establishing an "isomorphism" between specific applications and concepts in sheaf cohomology, enabling rigorous representation of economic and reinforcement learning systems in the language of sheaves.

\section{Feasibility and Focus}

I recognize the ambitious scope of this research. To make it tractable, I'll:

\begin{itemize}
    \item Identify the highest-impact subset of the broader conceptual schema with both theoretical significance and practical computational importance
    
    \item Leverage existing formalizations wherever possible to avoid reinventing the wheel
    
    \item Establish clear evaluation criteria for each phase to maintain focus
    
    \item Remain flexible about implementation approaches based on technical feasibility
\end{itemize}

While the complete integration of sheaf theory into multi-agent system modeling represents my ultimate vision, this PhD will focus on establishing the foundational framework and demonstrating its application in selected high-value contexts. I'll identify promising directions for future exploration beyond the scope of this dissertation.

\section{Timeline}

\begin{itemize}
    \item \textbf{Year 1:} Complete Phases 1 and 2
        \begin{itemize}
            \item Months 1-6: Survey and analysis of existing formalizations
            \item Months 7-12: Identification of target applications in economics and RL
        \end{itemize}
    \item \textbf{Year 2:} Complete Phase 3 and begin Phase 4
        \begin{itemize}
            \item Months 1-12: Development of sheaf-theoretic modeling framework
            \item Months 10-12: Begin formalization in dependently typed languages
        \end{itemize}
    \item \textbf{Year 3:} Complete Phases 4, 5, and 6, thesis writing
        \begin{itemize}
            \item Months 1-3: Complete formalization in dependently typed languages
            \item Months 4-9: Computational implementation and application
            \item Months 7-12: Evaluation and thesis writing
        \end{itemize}
\end{itemize}

\section{Conclusion}

This proposal charts a course through the unexplored territory where sheaf theory meets multi-agent systems. By bridging abstract mathematics with practical computational implementations, I aim to create a conceptual framework and technological toolset capable of transforming how we simulate and understand complex systems. If successful, this research will not only advance theoretical understanding but also provide practical tools for modeling, analyzing, and optimizing multi-agent systems across domains – from economic markets to artificial intelligence.

The increasing complexity of our world demands new mathematical frameworks. Sheaf theory, with its emphasis on local-to-global relationships, offers a promising path forward. I believe this research represents an opportunity to make a lasting contribution at the intersection of mathematics, economics, and computer science.

\chapter{Literature Review}

\section{Background and Motivation}
Multi-agent systems---collections of autonomous agents that must coordinate their actions---are frequently modeled using graphs. In a graph-based model, each agent is represented as a node and pairwise interactions or constraints between agents are represented as edges. Such graph models underpin many classical coordination algorithms: for instance, consensus protocols, distributed optimization, formation control, and flocking behaviors all typically assume an underlying communication graph among agents. Graph abstractions have proven effective for relatively homogeneous systems and simple interaction rules.

However, as we push toward more complex and heterogeneous multi-agent systems, graph models alone encounter limitations. A simple graph can only indicate which agents are connected, not the detailed \emph{type} of interaction or constraint on each connection. In many scenarios we need to encode high-dimensional data at nodes (agent states) and more complex relationships along edges (e.g. transformations, constraints, or tasks that relate the data of two agents). For example, in a team of heterogeneous robots, an edge between a ground robot and an aerial drone might carry a different kind of relationship (say, a sensor alignment or a communication relay constraint) than an edge between two ground robots. A standard graph with scalar edge weights cannot fully capture this difference in interaction structure.

To address these challenges, recent research has proposed using \textbf{cellular sheaves} to generalize graph-based coordination models. A cellular sheaf is a construct from algebraic topology that can be thought of as a graph augmented with structured data attached to nodes and edges, together with consistency maps that govern the relationships between node data and edge data. This topological framework allows encoding of richer relational structure among agents and subsystems than a simple graph alone can accommodate. By using sheaves, one can incorporate vector spaces of states at each node (potentially of varying dimensions for different agents) and encode constraints or transformation laws on each edge via linear maps between these vector spaces. Intuitively, the sheaf assigns local data spaces and specifies how local sections of data must agree on overlaps (edges), thereby generalizing the notion of "consensus" to more complex consistency conditions.

The motivation for studying multi-agent coordination on sheaves is to enable advanced decentralized coordination in modern applications like drone swarms, sensor networks, and heterogeneous robotic teams, where agents might have different roles and capabilities. In such settings, enforcing coordination involves maintaining certain consistency or alignment between high-dimensional state information, not just matching scalar values. The sheaf-theoretic approach provides a language to rigorously describe these consistency conditions and to design algorithms that ensure they are met. Recent work by Hanks \emph{et al.} introduced a general framework for heterogeneous multi-agent coordination using cellular sheaves and a new concept called the \textbf{nonlinear sheaf Laplacian}, which extends the familiar graph Laplacian to the sheaf setting and incorporates nonlinear interaction potentials \footnote{Tyler Hanks et al., "Distributed Multi-Agent Coordination over Cellular Sheaves," \textit{arXiv preprint arXiv:2504.02049} (2025). This paper introduces the cellular sheaf framework for coordination and the nonlinear sheaf Laplacian, using it to unify tasks like consensus, formation, and flocking under a single optimization framework.}.

In parallel, Riess \cite{RiessThesis2023} explored a related line of inquiry in his thesis on lattice-theoretic multi-agent systems, developing a \emph{discrete Hodge theory} for certain lattice-valued sheaves, including a Laplace operator (dubbed the \emph{Tarski Laplacian}) that generalizes the graph Laplacian to ordered or non-linear data domains\footnote{Hans Riess, \textit{Lattice Theory in Multi-Agent Systems} (Ph.D. thesis, University of Pennsylvania, 2023). Riess develops a Laplacian for lattice-valued (order-theoretic) sheaves and proves a Hodge-type theorem (the "Hodge-Tarski theorem") relating its fixed points to global sections of the sheaf. Chapter 5 of his thesis provides an introduction to sheaf theory from a combinatorial perspective, motivating the use of sheaves in multi-agent systems.}. These efforts reflect a broader trend of leveraging topological methods (sheaves, cohomology, Hodge theory) in the analysis of complex networks and coordination problems.

The use of cellular sheaves in multi-agent coordination promises:
\begin{itemize}
    \item \textbf{Richer Modeling Capability:} We can attach vector spaces (state spaces) to each agent and each interaction, allowing heterogeneity in what information is carried by different parts of the network. Constraints between agents can be represented by linear maps on these vector spaces, capturing directionality or transformations in the interaction.
    \item \textbf{Generalized Consensus and Coordination:} Classical consensus requires all agents' states (often scalars) to agree. Using sheaves, one can enforce more general consistency conditions (for instance, agents' states might need to satisfy a linear equation or align through a rotation matrix on each edge) rather than simple equality. This broadens the type of coordination objectives we can formalize.
    \item \textbf{Decentralized Algorithms:} The sheaf formalism, combined with tools like the sheaf Laplacian, yields natural generalizations of diffusion or gossip algorithms. As we will see, a \emph{sheaf Laplacian flow} can drive the system toward satisfying edge consistency constraints just as the graph Laplacian flow drives a network to consensus. Moreover, by formulating coordination as an optimization problem (a \emph{homological program}) on the sheaf, we can apply powerful distributed optimization methods (like ADMM) to solve it.
    \item \textbf{Connections to Learning:} Interestingly, these sheaf-based processes have analogies to recent Graph Neural Network (GNN) architectures. In a GNN, each layer updates node feature vectors by aggregating transformed neighbor features. In a sheaf-based coordination algorithm, each iteration updates agent states using neighbors' information and certain linear maps (the sheaf restrictions). This parallel suggests that \emph{sheaf neural networks} could harness the additional structure of sheaves to improve learning on graphs with complex relationships\footnote{See, e.g. "Sheaf Neural Networks," \textit{arXiv:2012.06333} (2020), and Cristian Bodnar et al., "Neural Sheaf Diffusion: A Topological Perspective on Heterophily and Oversmoothing," \textit{International Conference on Learning Representations (ICLR)} 2023. These works introduce neural network layers based on sheaf Laplacians, demonstrating advantages in handling directed or signed relationships (heterophily) on graphs.}.
\end{itemize}

The remainder of this paper is organized as follows. In Section 2, we introduce the key mathematical concepts: graphs, cellular sheaves on graphs, and sheaf Laplacians. We provide definitions and basic properties, drawing on established literature for a gentle introduction. Section 3 revisits the classical graph Laplacian and consensus dynamics, to build intuition for the sheaf generalization. Section 4 then presents the sheaf Laplacian in more detail and interprets its effect as a decentralized flow for enforcing consistency (a generalized consensus). Section 5 lays out a general framework for formulating multi-agent coordination problems using sheaves, including how to encode various coordination tasks. Section 6 summarizes the main contributions and benefits of this framework in comparison to classical approaches. In Section 7, we formally define the \emph{homological program} for multi-agent coordination, which is an optimization problem set up on a cellular sheaf. Section 8 details how to solve such programs using ADMM in a distributed manner; we derive the ADMM algorithm step by step and discuss its convergence. Section 9 draws a connection between the iterative solution (which can be seen as message-passing on the network) and neural network architectures, effectively viewing each iteration as a "layer" of a sheaf neural network. Based on this view, Section 10 discusses the advantages of incorporating sheaves into neural network models for learning, such as the ability to handle asymmetric relations and nonlinear propagation. Finally, we conclude with a brief discussion of future directions and applications.

\section{Key Concepts and Mathematical Foundations}

\subsection{Graphs and Data on Graphs}
We begin by reviewing the notion of a graph and the idea of attaching data to a graph. For our purposes, a graph $G$ is a pair $(V, E)$ consisting of a set of vertices (or nodes) $V$ and a set of edges $E$. Throughout, we will consider undirected graphs (suitable for modeling bidirectional communication between agents), and we assume $G$ is simple (no multiple edges between the same pair of vertices, and no self-loops). An edge $e \in E$ connecting vertices $i$ and $j$ will be denoted $e = \{i,j\}$ or simply $i$--$j$. We write $i \sim j$ if $\{i,j\} \in E$. Each vertex $i$ has a set of neighbors $N_i = \{\, j \in V : \{i,j\}\in E\,\}$.

In many applications, one associates data to the nodes or edges of a graph. For example, in a multi-agent system, each node $i$ might have an associated state vector $x_i \in \mathbb{R}^k$ describing the agent's internal variables (position, velocity, etc.), and each edge $e = \{i,j\}$ might represent a relationship or measurement involving agents $i$ and $j$. 

In the context of cellular sheaves on graphs, we can formalize these data assignments using terminology adapted from algebraic topology:

- A 0-\emph{cochain} on $G$ is an assignment of a vector to each vertex. If each vertex $i$ is assigned a vector $x_i \in \mathbb{R}^k$, we view $x = \{x_i\}_{i \in V}$ as a 0-cochain (with values in $\mathbb{R}^k$).

- A 1-\emph{cochain} on $G$ is an assignment of a vector to each edge. For instance, $y = \{y_e\}_{e \in E}$ with $y_e \in \mathbb{R}^m$ for each edge $e$ would be a 1-cochain (with values in $\mathbb{R}^m$).

If the vector spaces (like $\mathbb{R}^k$ or $\mathbb{R}^m$) are the same for all vertices or all edges, we call this a \emph{constant} assignment of data dimensions. In more general cases, one might have different types or dimensions of data at different nodes or edges. The mathematical structure that allows this generality (varying data spaces attached to different parts of a graph) is precisely a \emph{sheaf}, which we define next.

\subsection{Cellular Sheaves on Graphs}
Informally, a sheaf can be thought of as a way to attach data (or data spaces) to each part of a space (in our case, to each vertex and edge of a graph) along with consistency requirements that specify how data on neighboring parts must relate. While sheaves originated in algebraic geometry and topology for continuous spaces, here we focus on \emph{cellular sheaves on graphs}, which are a discrete version suitable for network systems \footnote{For an accessible introduction to cellular sheaves on graphs, see Jakob Hansen, "A gentle introduction to sheaves on graphs" (2020), available as a preprint. This source provides definitions and simple examples, oriented towards applications in engineering and network science.}.

\begin{definition}[Cellular Sheaf on a Graph]\label{def:sheaf}
Let $G=(V,E)$ be a graph. A \textbf{sheaf $\mathcal{F}$ on $G$} (sometimes called a \emph{cellular sheaf} when $G$ is viewed as a 1-dimensional cell complex) consists of the following:
\begin{itemize}
    \item For each vertex $i \in V$, a vector space $\mathcal{F}(i)$ over some field (here we take all vector spaces to be real Euclidean spaces). We call $\mathcal{F}(i)$ the \emph{stalk} of the sheaf at vertex $i$. Intuitively, $\mathcal{F}(i)$ is the data space associated with node $i$.
    \item For each edge $e=\{i,j\} \in E$, a vector space $\mathcal{F}(e)$ (the stalk at $e$). This is the data space associated with the interaction or relation between $i$ and $j$.
    \item For each incident pair $(i,e)$ with $i \in V$, $e \in E$, and $i$ incident to $e$ (meaning $i$ is an endpoint of edge $e$), a linear map 
    \[
    \mathcal{F}_{i \to e}: \mathcal{F}(i) \;\to\; \mathcal{F}(e),
    \] 
    called a \emph{restriction map}. If $e=\{i,j\}$, there are two such maps: $\mathcal{F}_{i\to e}$ mapping data from node $i$ to the edge $e$, and $\mathcal{F}_{j\to e}$ mapping data from node $j$ to the edge $e$. These maps specify how the local data at a node should be transformed or projected to be compatible with the data on the connecting edge.
\end{itemize}
A sheaf $\mathcal{F}$ assigns a vector space to each vertex and edge of $G$, and assigns to each inclusion of a vertex in an edge a linear map between the corresponding vector spaces.
\end{definition}

We often denote a sheaf by listing the data assignments, e.g. $\mathcal{F}(i)$ for vertices and $\mathcal{F}(e)$ for edges, with the understanding that the restriction maps $\mathcal{F}_{i\to e}$ are part of the data of the sheaf. The collection $\{\mathcal{F}_{i\to e}\}$ encodes the \textbf{consistency requirements} of the sheaf: a section of the sheaf (an assignment of data to every vertex and edge) is "valid" (a \emph{global section}) if whenever a vertex $i$ and an incident edge $e$ are considered, the data on $i$ matches the data on $e$ via the restriction map. Formally, a \textbf{global section} of $\mathcal{F}$ is a pair of families $(\{x_i \in \mathcal{F}(i)\}_{i\in V}, \{y_e \in \mathcal{F}(e)\}_{e\in E})$ such that for every edge $e=\{i,j\}$, we have 
\[
\mathcal{F}_{i\to e}(x_i) = y_e \quad \text{and}\quad \mathcal{F}_{j\to e}(x_j) = y_e.
\] 
In other words, the data $y_e$ on edge $e$ is the image of the data $x_i$ on vertex $i$ under the map to the edge, and similarly from the $j$ side. Equivalently $\mathcal{F}_{i\to e}(x_i) = \mathcal{F}_{j\to e}(x_j)$ for all edges $e=\{i,j\}$. This is the \textbf{local consistency condition} enforced by the sheaf.

If a choice of vertex data $\{x_i\}$ can be extended to some edge data $\{y_e\}$ satisfying these conditions, then $\{x_i\}$ is said to be a \emph{section at the vertices} (or 0-cochain) that is consistent. If no such $\{y_e\}$ exists, then $\{x_i\}$ has some inconsistency across at least one edge.

\begin{example}[Constant Sheaf]\label{ex:constant-sheaf}
A basic example is the \textbf{constant sheaf} on $G$ with fiber $\mathbb{R}^k$. This sheaf, denoted perhaps $\underline{\mathbb{R}^k}$, is defined by taking $\mathcal{F}(i) = \mathbb{R}^k$ for every vertex $i$ and $\mathcal{F}(e) = \mathbb{R}^k$ for every edge $e$, and letting each restriction map $\mathcal{F}_{i\to e}$ be the identity map on $\mathbb{R}^k$. In words, each node and each edge hold a $k$-dimensional vector, and the consistency condition is that for an edge $e=\{i,j\}$, the vectors at $i$, $j$, and $e$ must all agree (since $\mathcal{F}_{i\to e}$ and $\mathcal{F}_{j\to e}$ just send a vector to itself). A global section of the constant sheaf is thus an assignment of a vector $x_i \in \mathbb{R}^k$ to each node such that $x_i = x_j$ for every edge $\{i,j\}$. If the graph $G$ is connected, the only global sections are those where all nodes have the same vector (i.e. $x_1 = x_2 = \cdots = x_n$), which is exactly the usual consensus condition.
\end{example}

\begin{example}[Orientation Sheaf or Sign Sheaf]
Another instructive example is a sheaf that assigns $\mathbb{R}$ to each vertex and edge (so scalar data everywhere), but where the restriction maps might be $\pm 1$ instead of always $+1$. For instance, suppose we have a sheaf $\mathcal{F}$ with $\mathcal{F}(i)=\mathbb{R}$ for all $i$, $\mathcal{F}(e)=\mathbb{R}$ for all $e$, and define 
\[
\mathcal{F}_{i\to e} = 1 \quad \text{and}\quad \mathcal{F}_{j\to e} = -1
\] 
for an edge $e=\{i,j\}$. This means from the perspective of agent $i$, the edge's value should equal $x_i$, but from the perspective of agent $j$, the edge's value should equal \emph{minus} $x_j$. The consistency condition becomes $y_e = x_i = -x_j$. If we interpret $x_i$ as some scalar measurement at node $i$, this sheaf is encoding that measurements at $i$ and $j$ should be negatives of each other across the edge $e$. This could represent, e.g., a scenario where two agents must hold opposite voltages or opposite opinions. A global section here means an assignment of scalars to all nodes such that whenever $i$ and $j$ are connected, $x_j = -x_i$. On a graph that is a single edge with two vertices, the global sections are those pairs $(x_1,x_2)$ where $x_2=-x_1$. On a triangle graph, a global section would require $x_1=-x_2, x_2=-x_3, x_3=-x_1$, which forces $x_1=-x_2, x_2=-x_3, x_3 = x_1$ (by substituting), implying $x_1 = -x_2 = x_3 = -x_1$, so $x_1$ must be zero; thus only the zero assignment is a global section in that case (the sheaf was too restrictive to have a non-zero global solution on a cycle of odd length).
\end{example}

These examples illustrate that by adjusting the restriction maps, we can enforce different kinds of consistency: exact equality (as in the constant sheaf), equality up to a sign (orientation sheaf), or in general equality up to some linear transformation. In the most general case, if $\mathcal{F}_{i\to e}: \mathcal{F}(i)\to\mathcal{F}(e)$ is not surjective, the condition $\mathcal{F}_{i\to e}(x_i) = \mathcal{F}_{j\to e}(x_j)$ can be thought of as a system of linear equations relating $x_i$ and $x_j$. Sheaf theory provides a language to talk about solutions to all these local constraints collectively.

Before proceeding, we introduce some notation from algebraic topology that is useful to work with sheaves algebraically:
\begin{itemize}
    \item A \textbf{0-cochain} valued in the sheaf $\mathcal{F}$ is a choice of a vector $x_i \in \mathcal{F}(i)$ for each vertex $i$. We denote the space of all 0-cochains by 
    \[
    C^0(G;\mathcal{F}) := \bigoplus_{i\in V} \mathcal{F}(i).
    \] 
    This is essentially the direct product of all node stalks, i.e. the space of all assignments of node data.
    \item A \textbf{1-cochain} valued in $\mathcal{F}$ is a choice of a vector $y_e \in \mathcal{F}(e)$ for each edge $e$. The space of 1-cochains is 
    \[
    C^1(G;\mathcal{F}) := \bigoplus_{e\in E} \mathcal{F}(e).
    \]
    \item The \textbf{coboundary operator} (or restriction map at the level of cochains) is a linear map 
    \[
    \delta_{\mathcal{F}}: C^0(G;\mathcal{F}) \to C^1(G;\mathcal{F}),
    \] 
    defined by how it acts on a 0-cochain $x = \{x_i\}_{i\in V}$. For each edge $e=\{i,j\}$, the edge component of $\delta_{\mathcal{F}} x$ is given by 
    \[
    (\delta_{\mathcal{F}} x)_e = \mathcal{F}_{i\to e}(x_i) - \mathcal{F}_{j\to e}(x_j).
    \] 
    Here we must choose an orientation for each undirected edge to give a sign (we can pick an arbitrary orientation for each edge for the sake of this definition; the resulting $\delta_{\mathcal{F}}$ will depend on that choice but different choices are isomorphic). Intuitively, $\delta_{\mathcal{F}} x$ measures the \emph{disagreement} between the data at $i$ and $j$ when projected onto the edge $e$.
\end{itemize}

With this operator, the condition for $x$ to be extendable to a global section (with some $y$ on edges) is precisely $\delta_{\mathcal{F}} x = 0$. If $\delta_{\mathcal{F}} x = 0$, we say $x$ is a \textbf{cocycle} or that $x$ is \emph{closed} (borrowing terminology from cohomology theory). In plain terms, $\delta_{\mathcal{F}} x = 0$ means that for every edge $e=\{i,j\}$, $\mathcal{F}_{i\to e}(x_i) - \mathcal{F}_{j\to e}(x_j) = 0$, i.e., $\mathcal{F}_{i\to e}(x_i)=\mathcal{F}_{j\to e}(x_j)$. So $x$ is a section at the vertices that is consistent on all edges. Such an $x$ exactly corresponds to a global section of the sheaf (taking $y_e = \mathcal{F}_{i\to e}(x_i)$ for each edge $e$ yields a globally consistent assignment).

We call 1-cochains in the image of $\delta_{\mathcal{F}}$ \emph{exact} 1-cochains. These represent edge data that come from some assignment of node data. The kernel of $\delta_{\mathcal{F}}$ (i.e., all 0-cochains $x$ such that $\delta_{\mathcal{F}} x=0$) are the \emph{closed} 0-cochains, which correspond to global sections. In cohomological terms, $H^0(G;\mathcal{F}) := \ker(\delta_{\mathcal{F}})$ is the space of global sections of the sheaf $\mathcal{F}$. Since there is no 2-cochain space in a graph, the coboundary operator out of $C^1$ is zero, so its kernel is all of $C^1$. Therefore, $H^1(G;\mathcal{F}) := \frac{C^1(G;\mathcal{F})}{\text{im}(\delta_{\mathcal{F}})}$ measures the obstructions to being a global section (i.e., it represents cohomology classes of edge assignments). We will not delve deeply into cohomology here, but this viewpoint is useful in understanding certain results like Hodge decompositions.

To summarize, a sheaf on a graph gives rise to a linear operator $\delta_{\mathcal{F}}$ that measures local inconsistencies. This is analogous to the incidence matrix of a graph in classical graph theory: for a simple graph (with trivial sheaf), $\delta$ corresponds to an oriented incidence matrix, and $\delta x = 0$ means a scalar function $x$ on vertices is constant on each connected component. The sheaf case generalizes this by including linear transformations in the incidence relations.

\subsection{Sheaf Laplacians: Linear Case}
Given a sheaf $\mathcal{F}$ on $G$, we can construct an operator analogous to the graph Laplacian, but now acting on sections of the sheaf. The \textbf{sheaf Laplacian} is defined in a way similar to the combinatorial (graph) Laplacian, using the coboundary operator $\delta_{\mathcal{F}}$.

\begin{definition}[Sheaf Laplacian]
For a sheaf $\mathcal{F}$ on $G$, the (0-dimensional) \textbf{sheaf Laplacian} $L_{\mathcal{F}}: C^0(G;\mathcal{F}) \to C^0(G;\mathcal{F})$ is defined by 
\[
L_{\mathcal{F}} := \delta_{\mathcal{F}}^T \, \delta_{\mathcal{F}},
\] 
where $\delta_{\mathcal{F}}^T: C^1(G;\mathcal{F}) \to C^0(G;\mathcal{F})$ is the adjoint (transpose) of $\delta_{\mathcal{F}}$ with respect to the standard inner products on the Euclidean stalks (assuming each stalk $\mathcal{F}(i)$ and $\mathcal{F}(e)$ is an inner product space). 
\end{definition}

In coordinates, if we choose an orientation for each edge and bases for each stalk, $\delta_{\mathcal{F}}$ can be represented as a (matrix) operator mapping the vector of all node values to the vector of all edge differences. Then $L_{\mathcal{F}}$ is represented by the matrix $B^T B$ where $B$ is the matrix of $\delta_{\mathcal{F}}$. Thus $L_{\mathcal{F}}$ is symmetric positive semi-definite by construction (since $B^T B$ always is). More concretely, $L_{\mathcal{F}}$ acts on a 0-cochain $x = \{x_i\}_{i\in V}$ as follows: the $i$-th component $(L_{\mathcal{F}} x)_i$ is 
\[
(L_{\mathcal{F}} x)_i = \sum_{e : e=\{i,j\}\in E} \mathcal{F}_{i\to e}^T\big( \mathcal{F}_{i\to e}(x_i) - \mathcal{F}_{j\to e}(x_j) \big).
\]
This formula resembles the graph Laplacian in that it is a sum over neighbors $j$ of differences $x_i - x_j$, but here each difference is mapped into the edge space via $\mathcal{F}_{i\to e}$, and then mapped back via the adjoint of $\mathcal{F}_{i\to e}$. The result $(L_{\mathcal{F}}x)_i$ lies in the space $\mathcal{F}(i)$ (the same space as $x_i$), as expected for a Laplacian on the node space.

A convenient block-matrix way to write $L_{\mathcal{F}}$ is to arrange all node variables $x_i$ into a vector $x=(x_1;\ldots;x_{|V|})$ and similarly $y=(y_e)_{e\in E}$ for edge variables. Then $\delta_{\mathcal{F}}$ can be seen as a matrix with blocks $\mathcal{F}_{i\to e}$ and $-\mathcal{F}_{j\to e}$ in the rows corresponding to edge $e$ and columns corresponding to vertices $i$ and $j$. 

For vertices $i,j \in V$, define the block entries of the Laplacian matrix as:
\[
[L_{\mathcal{F}}]_{ij} = 
\begin{cases}
\sum_{e \ni i} \mathcal{F}_{i\to e}^T \mathcal{F}_{i\to e} & \text{if } i = j, \\
-\mathcal{F}_{i\to e}^T \mathcal{F}_{j\to e} & \text{if } e = \{i,j\} \in E, \\
0 & \text{otherwise}.
\end{cases}
\]

This generalizes the familiar formula for the graph Laplacian matrix $L = D - A$. In the special case where each $\mathcal{F}_{i\to e}$ is the identity (and scalar), $\mathcal{F}_{i\to e}^T \mathcal{F}_{i\to e} = 1$ and $\mathcal{F}_{i\to e}^T \mathcal{F}_{j\to e} = 1$, making $[L_{\mathcal{F}}]_{ii} = \deg(i)$ and $[L_{\mathcal{F}}]_{ij} = -1$ for adjacent vertices $i,j$. Indeed, we have:

\begin{lemma}[Recovery of Graph Laplacian]\label{lem:trivial-sheaf-laplacian}
If $\mathcal{F}$ is the constant sheaf on $G$ with 1-dimensional stalks and identity restriction maps (as in Example \ref{ex:constant-sheaf}), then the sheaf Laplacian $L_{\mathcal{F}}$ is exactly the ordinary graph Laplacian $L_G$ on $G$. That is, $(L_{\mathcal{F}}x)_i = \sum_{j: j\sim i}(x_i - x_j)$ for $x \in C^0(G;\mathcal{F})$.
\end{lemma}
\begin{proof}
Under a trivial sheaf, for each edge $e=\{i,j\}$, $\mathcal{F}_{i\to e}$ and $\mathcal{F}_{j\to e}$ are the identity on $\mathbb{R}$, so $\mathcal{F}_{i\to e}^T$ is also identity. Then 
\[
(L_{\mathcal{F}}x)_i = \sum_{e=\{i,j\}} (1)^T(1\cdot x_i - 1\cdot x_j) = \sum_{j: j\sim i}(x_i - x_j),
\] 
which is precisely the definition of the graph Laplacian applied to the function $x: V\to \mathbb{R}$. In matrix terms, $L_{\mathcal{F}}$ has $L_{ii} = \sum_{e\sim i} 1 = \deg(i)$ and $L_{ij} = -1$ if $i\sim j$, matching the standard Laplacian matrix $L_G = D - A$.
\end{proof}

Thus, sheaf Laplacians are a true generalization of graph Laplacians. One can alternatively think of them as graph Laplacians with a possibly different weight matrix on each edge (represented by the linear maps $\mathcal{F}_{i\to e}$). In spectral graph theory, one sometimes encounters \emph{matrix-weighted graphs} where an edge between $i,j$ is associated with a weight matrix $W_{ij}$ rather than a scalar weight. The sheaf Laplacian is essentially the Laplacian of a matrix-weighted graph, where $W_{ij} = \mathcal{F}_{i\to e}^T \mathcal{F}_{j\to e}$ for edge $e=\{i,j\}$. This viewpoint has been discussed in literature on generalized network Laplacians \footnote{See, Hansen and Ghrist, "Learning Sheaf Laplacians from Smooth Signals," \textit{Proc. ICASSP 2019}, where the goal is to infer sheaf restriction maps from data.}.

Several important properties of the graph Laplacian carry over to the sheaf Laplacian:
\begin{itemize}
    \item $L_{\mathcal{F}}$ is symmetric positive semidefinite. Thus all its eigenvalues are $\lambda \ge 0$.
    
    \item The kernel of $L_{\mathcal{F}}$ (the nullspace) consists of all $x$ such that $L_{\mathcal{F}}x = 0$. Because $L_{\mathcal{F}} = \delta_{\mathcal{F}}^T \delta_{\mathcal{F}}$, we have $x^T L_{\mathcal{F}} x = \|\delta_{\mathcal{F}} x\|^2$. Therefore, $L_{\mathcal{F}} x = 0$ if and only if $\delta_{\mathcal{F}} x = 0$. 
    
    This means that the kernel of the sheaf Laplacian equals the kernel of the coboundary operator:
    \[
    \ker(L_{\mathcal{F}}) = \ker(\delta_{\mathcal{F}})
    \] 
    
    As we've seen, $\ker(\delta_{\mathcal{F}})$ is precisely the space of global sections of the sheaf. For a connected graph with a constant sheaf, this means $\ker(L_G)$ is the one-dimensional space of constant vectors (the classical result for graph Laplacians). For a general sheaf, $\dim \ker(L_{\mathcal{F}})$ could be larger if the sheaf has multiple degrees of freedom for consistency. For example, the sign sheaf on a graph with an even cycle has a two-dimensional space of global sections.
\end{itemize}

The sheaf Laplacian $L_{\mathcal{F}}$ acts on 0-cochains (node assignments). There is also a \emph{1-Laplacian} which would act on 1-cochains ($L^{(1)}_{\mathcal{F}} = \delta \, \delta^T$), but in this paper we primarily care about the 0-Laplacian because it governs how node values evolve to satisfy constraints. In the context of multi-agent coordination, we are typically dealing with states at nodes and constraints at edges, so $L_{\mathcal{F}}$ on nodes is the relevant operator.

\paragraph{Sheaf Laplacian Dynamics.} A useful way to think about $L_{\mathcal{F}}$ is as a generator of a diffusion or consensus-like process. Consider the ODE:
\begin{equation}\label{eq:linear-sheaf-dynamics}
    \dot{x}(t) = -L_{\mathcal{F}}\, x(t),
\end{equation}
where $x(t) \in C^0(G;\mathcal{F})$ is a time-dependent assignment of states to nodes. This is a linear system of ODEs. Expanding $\dot{x}_i = - (L_{\mathcal{F}} x)_i$ using the earlier formula,
\[
\dot{x}_i = -\sum_{e=\{i,j\}} \mathcal{F}_{i\to e}^T\big(\mathcal{F}_{i\to e}(x_i) - \mathcal{F}_{j\to e}(x_j)\big).
\]
This can be interpreted as each agent $i$ adjusting its state $x_i$ in the direction that reduces the inconsistency on each incident edge. Specifically, for each neighbor $j$ of $i$, the term $\mathcal{F}_{i\to e}^T(\mathcal{F}_{i\to e}(x_i) - \mathcal{F}_{j\to e}(x_j))$ is the "force" exerted on $x_i$ due to the inconsistency between $i$ and $j$'s data when viewed on edge $e$. The sum over neighbors accumulates all these corrections. The negative sign indicates $x_i$ moves opposite to the gradient of inconsistency (hence reducing the difference). This is precisely a generalized consensus algorithm: if $x_i$ and $x_j$ are not consistent according to the sheaf, they will exchange information to become more aligned.

Because $L_{\mathcal{F}}$ is positive semidefinite, one can show that $\frac{d}{dt}\|x(t)\|^2 = - x(t)^T L_{\mathcal{F}} x(t) = - \|\delta_{\mathcal{F}} x(t)\|^2 \le 0$. Thus $\|x(t)\|^2$ is non-increasing. 

By LaSalle's invariance principle \cite{Khalil2002}, the system will converge to the largest invariant set where $\frac{d}{dt}\|x(t)\|^2 = 0$, which corresponds to the set where $\|\delta_{\mathcal{F}} x(t)\|^2 = 0$, or equivalently, $\delta_{\mathcal{F}} x(t) = 0$. This set is precisely $\ker(L_{\mathcal{F}})$, the space of global sections of the sheaf.

Therefore, for a connected graph and a sheaf with appropriate properties, the dynamics in equation \eqref{eq:linear-sheaf-dynamics} will converge to a global section $x(\infty)$ such that $L_{\mathcal{F}} x(\infty) = 0$, meaning all edge constraints are satisfied. The specific global section that the system converges to depends on the initial condition $x(0)$ and the dimension of $\ker(L_{\mathcal{F}})$\footnote{If $\ker(L_{\mathcal{F}})$ has dimension greater than 1, there are multiple possible equilibrium points. A simple example is standard consensus, where the system converges to all nodes having the same value, but that common value depends on the initial average.}.

In closing this subsection, we note that all the above has been the \emph{linear} theory of sheaves and Laplacians. The restriction maps $\mathcal{F}_{i\to e}$ were linear and the Laplacian is a linear operator. This suffices for modeling tasks where the constraint relationship between agent states is linear. But many coordination tasks are inherently nonlinear (e.g., maintaining a formation at a certain distance involves a quadratic constraint on positions, flocking involves nonlinear dynamics). To capture that, we move next to the concept of a \textbf{nonlinear homological program}, which introduces nonlinear edge potentials, effectively creating a \emph{nonlinear sheaf Laplacian} in the process.

\subsection{Nonlinear Homological Programs on Sheaves}\label{sec:nonlinear-homological-program}
The term "homological program" comes from earlier work \cite{HansenHomologicalProg2019} where optimization problems were formulated in terms of homology/cohomology of a network (for sensor network coordination). Here we adopt the formulation by Hanks \emph{et al.}\footnote{Hanks et al., 2025 (see previous footnote). In their framework, a \emph{nonlinear homological program} generalizes the earlier linear homological programs by allowing nonlinear objective and constraint functions while still leveraging the sheaf structure.} to define a general optimization problem that a multi-agent coordination task can be translated into. Roughly speaking, a \emph{nonlinear homological program} consists of:
\begin{itemize}
    \item A graph $G=(V,E)$ representing the communication topology among agents.
    \item A sheaf $\mathcal{F}$ on $G$ which specifies the data (state) spaces for each agent and each interaction and how those should match up.
    \item Convex objective functions $f_i: \mathcal{F}(i) \to \mathbb{R} \cup \{+\infty\}$ for each node (these represent local costs for each agent's own state; they can enforce constraints by taking value $+\infty$ on disallowed states).
    \item Convex \emph{potential functions} $U_e: \mathcal{F}(e) \to \mathbb{R} \cup \{+\infty\}$ for each edge (these represent the cost or constraint associated with the relation between the two endpoints' data when projected to the edge).
\end{itemize}

We allow extended real-valued objective and potential functions to handle constraints as well as costs (a common trick in convex optimization: a constraint $h(x)=0$ can be enforced by adding an indicator function $I_{\{h(x)=0\}}(x)$ which is 0 if $h(x)=0$ and $+\infty$ otherwise). In many cases, $U_e$ will be something like a strongly convex function that has a unique minimizer at the "allowed" relative configuration of $i$ and $j$'s data, and $f_i$ might be something like a regular quadratic penalty or other cost for agent $i$'s state deviating from a preferred value.

Given this data, we can define the optimization problem:
\begin{equation}\label{eq:homological-program}
\begin{aligned}
\min_{\substack{x \in C^0(G;\mathcal{F}) \\ y \in C^1(G;\mathcal{F})}} \quad & \sum_{i\in V} f_i(x_i) + \sum_{e \in E} U_e(y_e) \\
\text{s.t.}\quad & \mathcal{F}_{i\to e}(x_i) = y_e, \quad \forall e=\{i,j\}\in E, \; \forall i \text{ endpoint of } e.
\end{aligned}
\end{equation}
Here the decision variables are $x = \{x_i\}_{i\in V}$ (node assignments) and $y=\{y_e\}_{e\in E}$ (edge assignments), and the constraints enforce that $(x,y)$ forms a global section of the sheaf (i.e. that $y_e$ equals the restriction of $x$ on both endpoints of $e$). In other words, $y = \delta_{\mathcal{F}} x$ (with appropriate sign conventions, but effectively the constraint means $\delta_{\mathcal{F}} x = 0$ as in the linear case, except here we allow an $y_e$ which is then also optimized). If we eliminated $y$, we could equivalently write the problem as:
\[
\min_{x \in C^0(G;\mathcal{F})} \sum_{i} f_i(x_i) + \sum_{e} U_e\!\big(\mathcal{F}_{i\to e}(x_i)\ -\ \mathcal{F}_{j\to e}(x_j)\big),
\] 
with the understanding that $e=\{i,j\}$. This form shows it is an optimization over node variables, where each edge contributes a coupling cost $U_e$ of the difference (or discrepancy) between $x_i$ and $x_j$ as measured in $\mathcal{F}(e)$.

A more compact way to express \eqref{eq:homological-program} is using the cochain notation:
\begin{equation}\label{eq:homological-program-compact}
\begin{aligned}
\min_{x \in C^0(G;\mathcal{F})} \quad & \sum_{i\in V} f_i(x_i) + \sum_{e\in E} U_e\big((\delta_{\mathcal{F}}x)_e\big) \\
= \min_{x \in C^0(G;\mathcal{F})} \quad & \sum_{i} f_i(x_i) + U\big(\delta_{\mathcal{F}} x\big),
\end{aligned}
\end{equation}
where $U: C^1(G;\mathcal{F}) \to \mathbb{R}\cup\{\infty\}$ is the function that on a 1-cochain $y=\{y_e\}$ evaluates to $\sum_e U_e(y_e)$. The feasible set where $\delta_{\mathcal{F}} x = 0$ (global sections) can be enforced by making $U_e(0)$ finite (e.g. 0) and $U_e(y_e) = +\infty$ for $y_e \neq 0$ if we wanted a hard constraint of consistency. But more generally, allowing $U_e$ to be smooth and convex (say quadratic) penalizes inconsistency softly rather than hard-constraining it.

\begin{definition}[Nonlinear Homological Program]\label{def:nhp}
The optimization problem \eqref{eq:homological-program-compact} defined by graph $G$, sheaf $\mathcal{F}$, node objectives $\{f_i\}$, and edge potentials $\{U_e\}$ is called a \textbf{nonlinear homological program} (NHP), denoted 
\[
P = (V, E, \mathcal{F}, \{f_i\}, \{U_e\}).
\]
We say $P$ is \emph{convex} if each $f_i$ and $U_e$ are convex functions. We say $P$ is \emph{feasible} if there exists at least one global section (i.e. $\exists x$ with $\delta_{\mathcal{F}} x = 0$) in the domain of the functions such that the objective is not $+\infty$.
\end{definition}

When $P$ is convex, it is a convex optimization problem that can in principle be solved by centralized solvers. But our interest is in solving it \emph{in a distributed fashion}, leveraging the graph structure. The typical difficulty is the coupling introduced by the $\delta_{\mathcal{F}} x$ term: $\delta_{\mathcal{F}}x=0$ couples variables of neighboring nodes. However, as we will see, this coupling is sparse (each edge only couples two nodes), and we can exploit that using methods like ADMM.

\paragraph{Examples of Homological Programs.} We give a couple of examples to illustrate how classical multi-agent problems fit into this framework.

\begin{example}[Consensus as a Homological Program]
Let $G$ be a connected graph of $N$ agents. Each agent $i$ has a variable $x_i \in \mathbb{R}$ (scalar consensus problem). We can use the constant sheaf $\underline{\mathbb{R}}$ on $G$ (so $\mathcal{F}(i)=\mathbb{R}$, $\mathcal{F}(e)=\mathbb{R}$, and $\mathcal{F}_{i\to e}$ identity). For node objectives, take $f_i(x_i) = 0$ for all $i$ (the agents themselves don't have any private cost; we just want consensus). For edge potentials, take $U_e(y_e) = \frac{\rho}{2}y_e^2$ (a quadratic function penalizing differences), where $\rho>0$ is some weight. Then the program \eqref{eq:homological-program-compact} becomes 
\[
\min_{x \in \mathbb{R}^N} \sum_{e=\{i,j\}} \frac{\rho}{2} (\mathcal{F}_{i\to e}(x_i) - \mathcal{F}_{j\to e}(x_j))^2 = \min_{x} \frac{\rho}{2}\sum_{e=\{i,j\}} (x_i - x_j)^2.
\]
This objective expands to $\frac{\rho}{2}x^T L_G x$, where $L_G$ is the graph Laplacian. Because $L_G$ is positive semidefinite and has $\mathbf{1}$ (all-ones vector) in its nullspace, this objective is minimized (to 0) by any $x$ that is constant on each connected component of $G$. Thus the minimizers are exactly the consensus states $\{x_i = c, \, \forall i\}$ for some $c$. So this homological program encodes the consensus problem. 

If we wanted to force consensus to a particular value, we could add a node objective, e.g. $f_1(x_1) = I_{\{x_1 = c^*\}}(x_1)$ (an indicator forcing $x_1=c^*$) to pin one node to some $c^*$. Then the unique minimizer would be $x_i=c^*$ for all $i$. In distributed averaging or agreement problems, often each agent starts with some initial value and they converge to the average; that scenario can be formulated as a consensus optimization with additional structure (like preserving sum, which is linear and might be seen via Lagrange multipliers).
\end{example}

\begin{example}[Formation Control as a Homological Program]
Consider 3 robots that need to form a triangle of specified pairwise distances. Label them 1,2,3. The communication graph is complete $K_3$ (each pair can measure relative distance). For simplicity, assume in 2D plane. Let the state of robot $i$ be $x_i \in \mathbb{R}^2$ (its position). Use the constant sheaf $\underline{\mathbb{R}^2}$ (so that each edge also has $\mathbb{R}^2$ as the data space, and restriction maps identity, meaning we are expecting positions to match on edges for consistency—here we will handle desired offsets via potentials, not via the linear map). Now for each edge $e=\{i,j\}$, let $d_{ij}^*$ be the desired vector difference from $i$ to $j$. For instance, we want $x_j - x_i = d_{ij}^*$. If we set a potential 
\[
U_{e}(y_e) = \frac{k}{2}\|\,y_e - d_{ij}^*\,\|^2,
\] 
which is minimized when $y_e = d_{ij}^*$, then our program becomes:
\[
\min_{x_1,x_2,x_3} \frac{k}{2}\Big(\| (x_2 - x1) - d_{12}^*\|^2 + \| (x_3 - x1) - d_{13}^*\|^2 + \| (x_3 - x2) - d_{23}^*\|^2\Big).
\]
This sums squared errors for each pair relative to the desired difference. The minimizers (if $d^*_*$ are compatible) correspond to formations congruent (via translation) to the desired triangle. In fact, there will be a 2-dimensional family of minimizers corresponding to all global translations of the formation (and possibly rotations or reflections if the $d^*$ are only distances not oriented differences). Those correspond to non-uniqueness due to symmetry (global sections can differ by an element of $H^0$, which here includes a translation freedom in the plane).
\end{example}

Many other tasks (flocking, where edges enforce velocity alignment and distance constraints; distributed sensor localization; etc.) can be encoded similarly. The pattern is: each constraint or coupling between agents is encoded as an edge potential that penalizes deviation from the desired relation.

A key theoretical result is that if all $f_i$ are convex and all $U_e$ are convex, then $P$ is a convex optimization problem. If further each $U_e$ is differentiable and has some curvature (e.g., strictly or strongly convex), then standard results ensure nice properties like uniqueness of the minimizer. In particular:

\begin{theorem}[Convexity of Homological Program] 
If each $f_i: \mathcal{F}(i)\to \mathbb{R}\cup\{\infty\}$ is convex and each $U_e: \mathcal{F}(e)\to \mathbb{R}\cup\{\infty\}$ is convex (and lower bounded), then $P$ is a convex optimization problem. Moreover, if each $U_e$ is differentiable and strongly convex (i.e., has a unique minimizer $b_e$ for each $e$) and each $f_i$ is convex, then $P$ has a unique minimizer (assuming feasibility). 
\end{theorem}
\begin{proof}[Sketch of Proof]
The objective in \eqref{eq:homological-program-compact} is a sum of convex functions in $(x_i)$ and $(\delta_{\mathcal{F}} x)$ respectively, composed with a linear map $\delta_{\mathcal{F}}$. The constraints can be incorporated into the objective via indicator functions if needed. The sum of convex functions remains convex. Strong convexity of $U_e$ ensures the overall objective is strictly convex in directions that change any inconsistent component, so combined with convex $f_i$ (which could be just 0, not necessarily strictly convex), the whole objective becomes strictly convex in $x$ if the only directions in nullspace of second derivative correspond to $\delta_{\mathcal{F}} x = 0$ directions (which typically are finite-dimensional global symmetries). Under technical conditions, a unique minimizer modulo those symmetries exists (one may need to fix one reference point to break symmetry). Detailed proofs appear in Hanks et al.~\cite{HanksEtAl}.
\end{proof}

Given that $P$ is convex (in nice cases), the next question is: how do we solve it in a \emph{distributed} manner? In a network of agents, we want an algorithm where each agent only uses information from its neighbors to update its state, and yet the group converges to the global minimizer of $P$. This is where the \textbf{ADMM} (Alternating Direction Method of Multipliers) comes into play, which we cover in detail in Section \ref{sec:admm}.

Before diving into ADMM, we reflect on how the above formalism connects back to our original aims:
- The sheaf $\mathcal{F}$ provides the abstraction to encode various constraints and different data types on edges.

- The $f_i$ can encode individual objectives or constraints for each agent (for example, each robot has a cost for deviance from a nominal path, or a sensor has a cost for energy usage).

- The $U_e$ encode the coordination objectives (like consensus, formation shape maintenance, collision avoidance if desired via inequality constraints, etc.) between agents.

- The condition $\delta x = 0$ at optimum corresponds to all coordination constraints being satisfied (the network reaches a consistent configuration).

- If we set $f_i = 0$ and each $U_e$ as an indicator of the constraint, then solving $P$ is basically finding a global section of the sheaf that satisfies all constraints (if one exists). If $U_e$ are penalties rather than hard constraints, solving $P$ yields the "closest" approximate global section balancing the costs.

This optimization view is powerful, but solving it requires some method. We now have the stage set to apply a distributed optimization algorithm.

\subsection{ADMM for Distributed Optimization}\label{sec:admm}
The Alternating Direction Method of Multipliers (ADMM) is a well-known algorithm in convex optimization that is particularly useful for problems with separable objectives and linear constraints \cite{BoydADMM2011}. It decomposes a problem into subproblems that are then coordinated to find a global solution.

Before applying ADMM to our specific problem $P$, let's recall the standard form of ADMM. Consider a generic problem:
\[
\min_{u, v} \; F(u) + G(v) \quad \text{s.t.}\quad A u + B v = c,
\]
where $F$ and $G$ are convex functions in $u$ and $v$ respectively, and $A, B$ are matrices (or linear operators) and $c$ is some constant vector. The augmented Lagrangian for this problem is 
\[
L_\rho(u,v,\lambda) = F(u) + G(v) + \lambda^T (Au + Bv - c) + \frac{\rho}{2}\|Au + Bv - c\|^2,
\]
where $\lambda$ is the dual variable (Lagrange multiplier) and $\rho>0$ is a penalty parameter. ADMM proceeds by iteratively performing:
\begin{align*}
u^{k+1} &= \arg\min_u L_\rho(u, v^k, \lambda^k), \\
v^{k+1} &= \arg\min_v L_\rho(u^{k+1}, v, \lambda^k), \\
\lambda^{k+1} &= \lambda^k + \rho (A u^{k+1} + B v^{k+1} - c).
\end{align*}

These steps can be interpreted as follows: first, minimize the augmented Lagrangian with respect to $u$ while keeping $v$ fixed; then minimize with respect to $v$ while keeping the updated $u$ fixed; finally, update the multiplier $\lambda$ using a gradient step on the dual function.

ADMM is particularly effective when $F$ and $G$ lead to easier subproblems in $u$ and $v$ separately than the combined problem in $(u,v)$. Additionally, in multi-agent settings, the updates can often be computed in parallel when $u$ or $v$ comprise independent components corresponding to different agents.

In a multi-agent context, a common approach is to introduce local copies of shared variables and then enforce consistency via constraints, which ADMM handles. We will do exactly that: each agent will maintain a local version of its state (which it optimizes in its own $f_i$) and there will be "edge variables" that enforce the coupling, similar to $y_e$ we had. In fact, we already wrote $P$ with both $x$ and $y$, which is convenient for ADMM because it separates nicely: the sum of $f_i(x_i)$ and $U_e(y_e)$ is separable over $i$ and $e$ if $x$ and $y$ are treated as separate blocks of variables.

Let us set up $P$ in the consensus form for ADMM:
\[
\min_{x,z} \; \sum_{i\in V} f_i(x_i) + \sum_{e\in E} U_e(z_e) \quad \text{s.t.}\; \delta_{\mathcal{F}} x + (-I) z = 0.
\]
Here we consider $z = y$ as the variable for edges, and $-I z$ just means $z_e$ enters with a minus sign (so the constraint is $\delta_{\mathcal{F}} x - z = 0$ which is equivalent to $\delta_{\mathcal{F}} x = z$, matching $y$ we had). We identify $u := x$ and $v := z$ in the ADMM template, with $F(u) = \sum_i f_i(x_i)$ and $G(v) = \sum_e U_e(z_e)$. The constraint matrix $A$ corresponds to $\delta_{\mathcal{F}}$ acting on $x$, and $B$ corresponds to $-I$ acting on $z$, and $c$ is zero. Each of these pieces has structure:

- $\sum_i f_i(x_i)$ is separable by agent.

- $\sum_e U_e(z_e)$ is separable by edge.

- The constraint $Ax + Bz = 0$ couples $x_i$ and $z_e$ for each incidence.

Applying ADMM to this formulation, we introduce a dual variable $\lambda$ associated with the constraint $\delta_{\mathcal{F}} x - z = 0$. The dual variable lives in the space of 1-cochains (since the constraint is a 1-cochain equation). It can be interpreted as a vector $\lambda = \{ \lambda_e\}_{e\in E}$ where $\lambda_e \in \mathcal{F}(e)^*$ (same space as $z_e$ effectively) is a Lagrange multiplier enforcing consistency on edge $e$. Economically, $\lambda_e$ can be seen as the "price" of disagreement on edge $e$; physically in a multi-agent system, it might correspond to a force or tension associated with edge $e$ trying to enforce the constraint.

The ADMM update steps become:
\begin{align}
x^{k+1} &= \arg\min_{x \in C^0(G;\mathcal{F})} \Big( \sum_i f_i(x_i) + \frac{\rho}{2}\sum_{e} \|\mathcal{F}_{i\to e}(x_i) - \mathcal{F}_{j\to e}(x_j) - z^k_e + \frac{1}{\rho}\lambda^k_e \|^2 \Big). \label{eq:admm-x-update}\\
z^{k+1} &= \arg\min_{z \in C^1(G;\mathcal{F})} \Big( \sum_e U_e(z_e) + \frac{\rho}{2}\sum_e \| \mathcal{F}_{i\to e}(x^{k+1}_i) - \mathcal{F}_{j\to e}(x^{k+1}_j) - z_e + \frac{1}{\rho}\lambda^k_e \|^2 \Big). \label{eq:admm-z-update}\\
\lambda^{k+1}_e &= \lambda^k_e + \rho\big(\mathcal{F}_{i\to e}(x^{k+1}_i) - \mathcal{F}_{j\to e}(x^{k+1}_j) - z^{k+1}_e\big), \quad \forall e\in E.\label{eq:admm-dual-update}
\end{align}

We have written $\mathcal{F}_{i\to e}(x^{k+1}_i) - \mathcal{F}_{j\to e}(x^{k+1}_j)$ inside since that's $(\delta_{\mathcal{F}} x^{k+1})_e$. In practice, it's easier to combine terms: note that in \eqref{eq:admm-x-update}, the quadratic term couples $x_i$ and $x_j$. But we can observe that \eqref{eq:admm-x-update} actually breaks into separate minimizations for each $i$ due to separability: each $f_i(x_i)$ plus terms for each edge incident to $i$. Specifically, agent $i$ sees terms involving $x_i$ from each edge $e$ incident on $i$. If we gather those, the $x_i$-update is:
\[
x^{k+1}_i = \arg\min_{u \in \mathcal{F}(i)} f_i(u) + \frac{\rho}{2}\sum_{e: e=\{i,j\}} \| \mathcal{F}_{i\to e}(u) - (\mathcal{F}_{j\to e}(x^k_j) - z^k_e + \frac{1}{\rho}\lambda^k_e) \|^2.
\]
This is effectively a proximal update: it says agent $i$ should choose $x_i$ as a trade-off between minimizing its own cost $f_i(x_i)$ and being close to a value that satisfies the constraints for each neighboring edge (where neighbor $j$'s last state plus the dual adjustments set a target). If $f_i$ is simple (e.g. quadratic), this is often an easily computable step (like averaging or thresholding).

\begin{definition}[Proximal Operator]
Given a proper, closed, and convex function $f: \mathcal{X} \to \mathbb{R} \cup \{+\infty\}$ and a positive parameter $\rho > 0$, the \emph{proximal operator} of $f$ with parameter $\rho$ is defined as:
\[
\mathrm{prox}_{\rho f}(v) := \arg\min_{x \in \mathcal{X}} \left\{ f(x) + \frac{1}{2\rho}\|x - v\|^2 \right\}
\]
The proximal operator can be interpreted as finding a point that balances minimizing $f$ while staying close to $v$. It generalizes the notion of projection onto a set (when $f$ is an indicator function of that set) and includes gradient steps as a special case (when $f$ is differentiable).
\end{definition}

The $z$-update in \eqref{eq:admm-z-update} is separable by edge $e$:
\[
z^{k+1}_e = \arg\min_{w \in \mathcal{F}(e)} U_e(w) + \frac{\rho}{2}\| w - (\mathcal{F}_{i\to e}(x^{k+1}_i) - \mathcal{F}_{j\to e}(x^{k+1}_j) + \frac{1}{\rho}\lambda^k_e) \|^2,
\]
which is the proximal operator of $U_e$ at the "point" $\mathcal{F}_{i\to e}(x^{k+1}_i) - \mathcal{F}_{j\to e}(x^{k+1}_j) + \frac{1}{\rho}\lambda^k_e$. If $U_e$ is an indicator of some feasible set or has a closed form proximal, this can often be solved locally by the two agents $i,j$ sharing the value.

Finally, the $\lambda$ update is just accumulation of residual $\delta_{\mathcal{F}} x - z$. Each $\lambda_e$ is updated by the two agents on edge $e$ exchanging their $x$ values and the current $z$ and updating $\lambda_e$.

The ADMM algorithm allows each agent and each edge to update based only on local information:

- Agent $i$ needs $x_j^k$ and $\lambda_e^k$, $z_e^k$ from edges to neighbors $j$.

- Edge $e=\{i,j\}$ needs $x_i^{k+1}$, $x_j^{k+1}$, and $\lambda_e^k$ to update $z_e$.

- Then edge $e$ also updates $\lambda_e$ with new $x$ and $z$.

Thus communication occurs along edges only, making this distributed.

However, a challenge arises: the $z$-update requires solving $\arg\min_w U_e(w) + \frac{\rho}{2}\|w - \xi\|^2$ where $\xi = (\mathcal{F}_{i\to e}(x_i^{k+1}) - \mathcal{F}_{j\to e}(x_j^{k+1}) + \frac{1}{\rho}\lambda_e^k)$. If $U_e$ is strongly convex and smooth, one could take a gradient step or closed form. But to keep it distributed, note that solving this may not be trivial in closed form unless $U_e$ is simple (like quadratic).

\subsection{Background on Nonlinear Operators and Fixed-Point Methods}
Before proceeding with the nonlinear sheaf Laplacian, we briefly review some mathematical concepts that will be essential for understanding the following derivations.

\paragraph{Subdifferentials and Non-smooth Optimization.} When dealing with non-differentiable convex functions, the standard notion of gradient is replaced by the \emph{subdifferential}. For a convex function $f: \mathcal{X} \to \mathbb{R} \cup \{+\infty\}$, the subdifferential at a point $x$, denoted $\partial f(x)$, is the set of all subgradients:
\[
\partial f(x) = \{g \in \mathcal{X}^* \mid f(y) \geq f(x) + \langle g, y-x \rangle \text{ for all } y \in \mathcal{X}\}
\]
where $\mathcal{X}^*$ is the dual space of $\mathcal{X}$. When $f$ is differentiable at $x$, then $\partial f(x) = \{\nabla f(x)\}$, a singleton set containing just the gradient. The notation $0 \in \partial f(x)$ indicates that $x$ is a minimizer of $f$, which is the first-order optimality condition for convex functions.

\paragraph{Gradient Flows and Dynamical Systems.} A gradient flow for minimizing a differentiable function $f$ can be expressed as the ODE:
\[
\dot{x} = -\nabla f(x)
\]
This system evolves in the direction of steepest descent of $f$. For a strongly convex function, this flow converges to the unique minimizer of $f$. When $f$ is not differentiable, we can use the generalized gradient flow:
\[
\dot{x} \in -\partial f(x)
\]
which is a differential inclusion rather than a differential equation.

\paragraph{Fixed-Point Iterations.} For finding a point $x^*$ such that $x^* = T(x^*)$ for some mapping $T$, fixed-point iteration methods use the update rule:
\[
x^{k+1} = T(x^k)
\]
Under certain conditions (like $T$ being a contraction), this iteration converges to a fixed point. Many optimization algorithms, including proximal methods, can be formulated as fixed-point iterations. The equation $x = T(x)$ is called a fixed-point equation.

\paragraph{Implicit and Explicit Methods.} When discretizing differential equations, explicit methods (like forward Euler) compute the next state directly from the current state, while implicit methods (like backward Euler) require solving an equation involving the next state. For gradient flows, the implicit Euler discretization leads to the proximal point algorithm, which has superior stability properties compared to explicit methods.

With these concepts in mind, we now return to our ADMM algorithm for distributed coordination and examine how the nonlinear sheaf Laplacian arises naturally from the $z$-update step.

\subsection{Nonlinear Sheaf Diffusion for the Edge Update}
Hanks \textit{et al.} provide a key insight: instead of solving the $z$-update in one shot, one can run an iterative \emph{diffusion process} (like our earlier $\dot{x} = -L_{\mathcal{F}} x$ but now nonlinear) to converge to the minimizer. This is where the \textbf{nonlinear sheaf Laplacian} concept enters. The $z$-update condition (optimality for each edge) set to zero is:
\[
0 \in \nabla U_e(z^{k+1}_e) + \rho\big( z^{k+1}_e - (\mathcal{F}_{i\to e}(x^{k+1}_i) - \mathcal{F}_{j\to e}(x^{k+1}_j) + \tfrac{1}{\rho}\lambda^k_e )\big).
\]
Rearranged:
\[
z^{k+1}_e = \mathcal{F}_{i\to e}(x^{k+1}_i) - \mathcal{F}_{j\to e}(x^{k+1}_j) + \tfrac{1}{\rho}\lambda^k_e - \tfrac{1}{\rho}\nabla U_e(z^{k+1}_e).
\]
This looks like a fixed-point equation for $z_e$. If we define $b_e$ to be the unique minimizer of $U_e$ (since $U_e$ strongly convex, assume $b_e = \arg\min U_e$ exists and $U_e(b_e)=0$ by shifting), then one expects $z^{k+1}_e$ to be somewhat in between the current discrepancy and $b_e$. 

Rather than directly solve it, they propose to treat $z$ not as an explicit variable to update but to integrate a dynamics that reaches this solution. Concretely, consider the continuous time dynamics:
\[
\dot{z}_e = -\nabla U_e(z_e) - \rho( z_e - (\mathcal{F}_{i\to e}(x_i) - \mathcal{F}_{j\to e}(x_j) + \tfrac{1}{\rho}\lambda_e) ).
\]
At equilibrium ($\dot{z}_e=0$), this yields exactly the above optimality condition for $z_e$. And crucially, this dynamic is local (involves agent $i$ and $j$ and the edge $e$ state). Moreover, one can show it converges to the optimum $z_e$ (since it's like a gradient descent in $w$ with step $\rho$ on a strongly convex function). In fact, if we embed this into the updates more systematically, it becomes a two-timescale algorithm: agents update $x$ and $\lambda$ at discrete steps, and edge states $z$ are driven by a faster dynamics to near the optimum.

Hanks \textit{et al.} go further to eliminate the need to explicitly represent $z_e$ at all: they incorporate this dynamics directly into $x$ update by noticing that eliminating $z$ and $\lambda$ yields a second-order update purely in $x$ which is equivalent to a \emph{nonlinear diffusion equation on the sheaf}. They call this the \textbf{nonlinear Laplacian dynamics}:
\begin{equation}\label{eq:nonlinear-laplacian-dyn}
\dot{x}(t) = -L_{\nabla U}^{\mathcal{F}} x(t),
\end{equation}
where $L_{\nabla U}^{\mathcal{F}}$ is the \emph{nonlinear sheaf Laplacian} defined via the relation 
\[
(L_{\nabla U}^{\mathcal{F}} x)_i = \sum_{e=\{i,j\}} \mathcal{F}_{i\to e}^T \nabla U_e\big(\mathcal{F}_{i\to e}(x_i) - \mathcal{F}_{j\to e}(x_j)\big).
\]
Compare this with the linear case earlier: previously, $(L_{\mathcal{F}}x)_i = \sum_e \mathcal{F}_{i\to e}^T(\mathcal{F}_{i\to e}(x_i) - \mathcal{F}_{j\to e}(x_j))$. Now we simply replace the linear difference by $\nabla U_e$ of that difference. If $U_e(y) = \frac{1}{2}\|y\|^2$, then $\nabla U_e(y)=y$ and we recover linear $L_{\mathcal{F}}$. If $U_e$ is, say, quadratic with minimizer $b_e$, $\nabla U_e(y) = K (y - b_e)$ for some $K$, and the dynamics tries to push $\mathcal{F}_{i\to e}(x_i) - \mathcal{F}_{j\to e}(x_j)$ towards $b_e$ on each edge.

Hanks \textit{et al.} proved that under the assumption that each $U_e$ is strongly convex with minimizer $b_e \in \operatorname{im}\delta_{\mathcal{F}}$ (meaning the desired edge outcomes are consistent with some global assignment), the nonlinear Laplacian dynamics \eqref{eq:nonlinear-laplacian-dyn} converges to a global section $x^*$ that achieves $\mathcal{F}_{i\to e}(x^*_i) - \mathcal{F}_{j\to e}(x^*_j) = b_e$ for each edge (essentially projecting the initial state onto the feasible set of constraints) \footnote{Hanks et al., Theorem 2, 2025. In our notation, it states that if $b = \{b_e\}_{e\in E}$ is an edge assignment in the image of $\delta_{\mathcal{F}}$ (so constraints are consistent), then $\dot{x}=-L_{\nabla U}^{\mathcal{F}} x$ converges to a point $x^\infty$ with $\delta_{\mathcal{F}} x^\infty = b$. Equivalently, $x^\infty$ is the global section achieving the desired $b_e$ on each edge. If $b$ is not exact (in image of $\delta_{\mathcal{F}}$), then no perfect global section exists; one can show the dynamics converges to a least-squares approximate solution in that case (projection onto $\operatorname{im}\delta_{\mathcal{F}}$).}.

Algorithmically, we can integrate a discrete approximation of \eqref{eq:nonlinear-laplacian-dyn} as the method for computing the projection in the $z$-update. This yields a fully distributed algorithm where agents continuously adjust their states by exchanging information with neighbors and applying the "forces" $\nabla U_e$.

\subsection{Summary of the Distributed Algorithm}
The ADMM-based distributed solution for the homological program can be described at a high level:
\begin{enumerate}
    \item Each agent $i$ initializes its state $x_i(0)$ (perhaps to some initial guess or measured value).
    \item \textbf{Repeat until convergence:}
    \begin{enumerate}
        \item Each agent performs a local state update (this corresponds to the $x$-update of ADMM, often just computing a gradient step of $f_i$ plus neighbor terms).
        \item The network performs a \emph{sheaf diffusion step}: neighbors $i,j$ communicate and adjust their states to reduce inconsistency via something like $x_i \leftarrow x_i - \alpha \mathcal{F}_{i\to e}^T(\mathcal{F}_{i\to e}(x_i) - \mathcal{F}_{j\to e}(x_j))$ (this is a discrete Euler step of the nonlinear Laplacian flow, repeated a few times).
        \item Agents update dual variables (or equivalently accumulate the inconsistency error).
    \end{enumerate}
\end{enumerate}
Under convexity and proper parameter tuning, this process converges to the optimal solution $x^*$ of the original problem $P$. The resulting $x^*$ is then a set of states each agent can take that jointly minimize the global objective while satisfying the constraints approximately (the error goes to zero if $b$ is reachable, or is minimized in least squares otherwise).

To ensure convergence, one often requires strong convexity or a sufficiently small step size in the diffusion. Hanks \textit{et al.} also discuss a relaxation where if strong convexity is absent, one can still run essentially the same algorithm but may need to modify the projection step (for instance, if $U_e$ has multiple minimizers, any point in the minimizer set could work, or add a tiny quadratic regularization to $U_e$ to select one).

\section{The Classical Graph Laplacian and Consensus Review}
Before further analyzing the sheaf-based method and exploring its connection to neural networks, we briefly review the classical graph Laplacian and consensus problem to anchor our intuition.

\subsection{Definition and Basic Properties of Graph Laplacian}
Let $G=(V,E)$ be a graph with $|V|=n$. The (combinatorial) \textbf{graph Laplacian} $L_G$ is the $n\times n$ matrix defined by
\[
(L_G)_{ij} = \begin{cases}
\deg(i) & \text{if } i=j,\\
-1 & \text{if } i\sim j \text{ (edge between $i,j$)},\\
0 & \text{otherwise}.
\end{cases}
\]
In other words, $L_G = D - A$ where $D$ is the diagonal matrix of degrees and $A$ is the adjacency matrix. Equivalently, for a vector $x \in \mathbb{R}^n$, the $i$-th component $(L_G x)_i = \sum_{j: j\sim i} (x_i - x_j)$. This operator is symmetric and positive semidefinite. Its nullspace is the span of the all-ones vector $\mathbf{1}$ (assuming $G$ is connected), meaning $L_G x = 0$ if and only if $x_1=x_2=\cdots=x_n$ (consensus).

Key properties include:
\begin{itemize}
    \item The eigenvalues $0=\lambda_1 \le \lambda_2 \le \cdots \le \lambda_n$ of $L_G$ relate to connectivity (Fiedler value $\lambda_2 >0$ if and only if $G$ is connected).
    \item The Laplacian can be seen as the matrix of quadratic form $\frac{1}{2}\sum_{i\sim j}(x_i - x_j)^2$ in the sense that $x^T L_G x = \sum_{\{i,j\}\in E} (x_i - x_j)^2$.
    \item The pseudoinverse $L_G^+$ exists for connected graphs and projects onto the subspace orthogonal to $\mathbf{1}$.
\end{itemize}

\subsection{Consensus Dynamics}
The simplest distributed coordination algorithm on a graph is the consensus protocol: each agent $i$ has a scalar $x_i(t)$, and they update according to 
\[
\dot{x}_i(t) = -\sum_{j \in N_i} (x_i(t) - x_j(t)).
\]
In vector-matrix form, $\dot{x} = -L_G x$. As mentioned, this is exactly of the form \eqref{eq:linear-sheaf-dynamics} for a trivial sheaf. It is well known that for a connected graph, $x(t) \to c\mathbf{1}$ as $t\to \infty$, where $c$ is the average of the initial values (if the system is weight-balanced or we use $\frac{1}{\deg(i)}$ weights, it converges to the average; for the simple Laplacian as given, it converges to some consensus value which, due to the design, actually equalizes but does not preserve average unless the graph is regular or one uses weights). Nonetheless, the main point is convergence to agreement.

This consensus can be seen as solving the optimization $\min \sum_{i<j,\, i\sim j}(x_i-x_j)^2$, which is minimized by $x_i = x_j$ for all edges (consensus). We solved it by a subgradient method (gradient of that quadratic cost yields $-Lx$).

\subsection{Example: Average Consensus}
Suppose 4 agents on a line graph have initial values $[3, 1, 4, 2]$. Running $\dot{x}=-Lx$, or in discrete form $x_i \leftarrow x_i + \epsilon\sum_{j \in N_i}(x_j - x_i)$, they will converge to $[2.5,2.5,2.5,2.5]$ (the average) in the limit. This illustrates Laplacian flow equalizing values.

\subsection{Graph Laplacian and Sheaf Laplacian Connection}
As we proved in Lemma \ref{lem:trivial-sheaf-laplacian}, a graph Laplacian is a special case of a sheaf Laplacian. Conversely, every sheaf Laplacian is a graph Laplacian of a certain expanded graph (each edge in original graph could be thought of as multiple edges connecting the vector components). This connection means results from algebraic graph theory, such as spectral bounds and convergence rates, have analogues in sheaf Laplacians (often depending on the norms of restriction maps).

One such analogy: the consensus algorithm's convergence speed is related to $\lambda_2(L_G)$. For sheaf Laplacians, the convergence of $\dot{x}=-L_{\mathcal{F}}x$ is related to the smallest nonzero eigenvalue of $L_{\mathcal{F}}$. Sheaf Laplacians can have multiple zero eigenvalues if the sheaf has multi-dimensional global sections, but if we mod out those, the spectral gap gives the rate at which local inconsistencies dissipate.

This ties into multi-agent coordination: a larger spectral gap means faster agreement (or constraint satisfaction) across the network. Using appropriate sheaf design can increase this gap by effectively adding "more edges" through vector dimensions coupling (like a high-dimensional edge constraints can act like multiple parallel constraints).

\section{Sheaf Laplacians and Decentralized Flows}
In Section 2, we discussed how $\dot{x}=-L_{\mathcal{F}}x$ acts as a decentralized flow driving the network toward global consistency (a sheaf global section). We now delve deeper into interpreting this in the context of multi-agent systems and mention how nonlinear extensions (using $\nabla U$) serve as decentralized controllers.

\subsection{Linear Sheaf Laplacian = Decentralized Linear Controller}
Consider a scenario where agents' goal is to satisfy linear equations among their states. For example, agent $j$ should have a state that equals a fixed linear transformation of agent $i$'s state. If that transformation is encoded in $\mathcal{F}_{i\to e}$ and $\mathcal{F}_{j\to e}$, then $L_{\mathcal{F}}x=0$ are exactly the conditions for all those equations to hold. The protocol $\dot{x}=-L_{\mathcal{F}}x$ is then a proportional control law: each agent adjusts based on weighted differences between the left and right sides of each equation for each constraint it's involved in. The fact that this is negative feedback on the constraint error ensures stability and convergence to satisfaction of the constraints (if possible).

For example, suppose two UAVs must maintain the same altitude difference as two UGVs (ground vehicles). We can set a sheaf where UAV altitude and UGV altitude on an edge have restriction map identity (so edge data is a "height difference"). The linear consensus tries to equalize those height differences, effectively making the UAVs and UGVs coordinate altitudes. This is a bit abstract, but basically any linear decentralized control can be seen as a Laplacian on an appropriate sheaf.

\subsection{Nonlinear Sheaf Laplacian = Decentralized Nonlinear Controller}
When constraints are nonlinear (like distance = 5, which is $\|x_i - x_j\|=5$), a common approach is to design a decentralized controller using gradient descent on a potential function encoding that constraint (for distance, a potential like $(\|x_i-x_j\|^2 -25)^2$). The nonlinear sheaf Laplacian does exactly this in a structured way: $\nabla U_e(\mathcal{F}_{i\to e}(x_i) - \mathcal{F}_{j\to e}(x_j))$ is the gradient of the edge's potential with respect to its arguments, pulling $x_i$ and $x_j$ in directions to satisfy the constraint. And $\mathcal{F}_{i\to e}^T$ maps that "force" back into the direction in $x_i$'s own space (like projecting a force vector from edge coordinate frame to agent's coordinate frame).

Thus $\dot{x}=-L_{\nabla U}^{\mathcal{F}} x$ is nothing but a network of agents each feeling forces from each incident constraint, where each constraint exerts equal and opposite forces on its two participants (ensuring consistency with Newton's third law, in a sense, if $\mathcal{F}_{i\to e}^T$ and $\mathcal{F}_{j\to e}^T$ produce opposite effects).

This perspective bridges to physical analogies: one can think of each edge constraint as a nonlinear spring between some linear functions of agent states. Then $L_{\nabla U}^{\mathcal{F}} x = 0$ is the static equilibrium of those springs, and the dynamics is a damped relaxation to equilibrium.

In multi-agent coordination, many problems (consensus, formation, flocking) can be interpreted as such potential-minimizing systems:

- Consensus: spring potential $\propto (x_i - x_j)^2$ along each communication link.

- Formation: spring potential $\propto (\|x_i-x_j\| - d_{ij})^2$ to maintain distance.

- Flocking (alignment): potential to align velocities (like an edge potential $\propto \|v_i - v_j\|^2$).

- And so on.

Therefore, the nonlinear sheaf Laplacian flow provides a unified model for these as gradient systems (which are easier to analyze for stability etc.).

\section{Framework for Multi-Agent Coordination}
Now we step back and outline a systematic procedure to model and solve a multi-agent coordination problem using the tools discussed.

\subsection{System Model}
We consider $N$ agents labeled $1,\ldots,N$. Each agent $i$ has:
\begin{itemize}
    \item A state vector $x_i \in \mathbb{R}^{n_i}$ (dimension $n_i$ could differ per agent).
    \item Dynamics or control authority: in this paper's scope, we assume we can directly set the state or the state is static during a coordination solving phase (i.e., we focus on the steady-state or algebraic coordination problem rather than transient dynamics). If dynamic, one might incorporate that into the sheaf as a time-extended sheaf (see Example 4 in Section 2.3 where a trajectory was a section on a path graph).
    \item Possibly an individual objective or cost function $f_i(x_i)$ capturing preferences (like how far from some nominal state).
\end{itemize}
The inter-agent interactions are given by a communication/interaction graph $G=(V,E)$ where each edge $\{i,j\}$ indicates agent $i$ and $j$ can directly exchange information or influence each other.

Crucially, to capture heterogeneous interactions, we introduce a sheaf $\mathcal{F}$ on $G$. How to design $\mathcal{F}$? Typically:
\begin{itemize}
    \item Choose $\mathcal{F}(i) = \mathbb{R}^{n_i}$, matching the agent's state space.
    \item For each edge $e=\{i,j\}$, decide what relationship it enforces. For example:
    \begin{itemize}
        \item If it's a consensus-type link for some component of state, $\mathcal{F}(e)$ could be a subspace of $\mathbb{R}^{n_i}$ and $\mathbb{R}^{n_j}$ that represents the quantity to agree on. $\mathcal{F}_{i\to e}, \mathcal{F}_{j\to e}$ would be projection maps onto that subspace. (E.g., if agents have $(position, temperature)$ as state and the edge is to agree on temperature, then $\mathcal{F}(e)=\mathbb{R}$ and $\mathcal{F}_{i\to e}$ picks the temperature coordinate from $\mathbb{R}^{n_i}$).
        \item If it's a formation link ensuring a certain relative position, $\mathcal{F}(e)$ might equal $\mathbb{R}^d$ (the physical space) and $\mathcal{F}_{i\to e}(x_i) = P_{ij} x_i$ is some projection of agent $i$'s state to a position (maybe identity if $x_i$ itself is position, or perhaps something like if $x_i$ includes orientation, we might want relative orientation, etc).
        \item If it's a constraint that agent $i$'s output should feed into agent $j$'s input (like in a supply-demand network or task scheduling), $\mathcal{F}_{i\to e}$ could be the output scalar and $\mathcal{F}_{j\to e}$ the input scalar, and $\mathcal{F}(e)=\mathbb{R}$ to equate them.
    \end{itemize}
\end{itemize}
This step is perhaps the most artful: one must decide how to represent the multi-agent coordination requirements in terms of local constraint mappings. Tools like a requirement graph or knowledge of the task are needed.

Once $\mathcal{F}$ is set, any assignment of states to agents can be checked for consistency via $\delta_{\mathcal{F}} x$. If $\delta_{\mathcal{F}} x = 0$, the agents perfectly satisfy all the relations encoded by edges.

\subsection{Coordination Objective}
Sometimes, just satisfying constraints is not enough; we might also optimize a performance criterion. For example, agents might want to minimize energy usage while maintaining formation, or minimize deviation from their preferred positions while reaching consensus.

This is where the node objective $f_i(x_i)$ come in. We gather those from the problem statement. These could be:

- Quadratic penalties $(1/2)\|x_i - \hat{x}_i\|^2$ for straying from some reference $\hat{x}_i$.

- Indicators or constraints like $x_i$ must lie in some set.

- 0 if no specific cost for agent aside from constraints.

Additionally, sometimes there are global objectives like maximize the average of something or minimize total variance. If they decompose per agent or can be assigned to edges, we incorporate accordingly (if truly global in a non-separable way, ADMM still can handle if represented as a star graph or by introducing extra variables, but we won't complicate with that here).

\subsection{General Procedure}
Given the above, the steps are:
\begin{enumerate}
    \item \textbf{Model with a sheaf:} Define $G$ and $\mathcal{F}$ to encode agent capabilities and required relationships.
    \item \textbf{Set up the homological program:} Write $P=(V,E,\mathcal{F},\{f_i\},\{U_e\})$. For each edge constraint, choose a suitable potential function $U_e$. If the task is a hard constraint (like exactly maintain distance $d$), one can use a barrier or a very stiff quadratic around $d$. If soft, then softer penalty.
    \item \textbf{Ensure convexity or handle non-convexity:} Many coordination tasks (like distance constraints) lead to non-convex $U_e$ (a quadratic $\|x_i-x_j\|^2$ is convex in $(x_i,x_j)$, but if we want $\|x_i-x_j\|=d$ exactly, that's non-convex or we could square it and penalize differences from $d^2$ which is convex in $x_i,x_j$ but has multiple minima including potentially $x_i=x_j$ and so on). In general, if the problem is not convex, ADMM still can be applied as a heuristic but no guarantee. Hanks et al. mostly considered convex (or convexified) cases. For exposition we lean on convex examples.
    \item \textbf{Apply ADMM (distributed):} Use the algorithm from Section 8 to solve. Translate it into pseudocode for agents: each agent iteratively updates its state and communicates with neighbors. Alternatively, derive the closed-form dynamics if possible.
    \item \textbf{Execute the strategy:} If this is an online control scenario, agents would continuously update their control inputs according to the solving algorithm and eventually reach coordination.
    \item \textbf{Validate constraints:} At convergence, verify if $\delta x \approx 0$ (all constraints satisfied within tolerance) and evaluate the cost.
\end{enumerate}

This structured approach helps ensure no constraint is forgotten and every multi-agent interaction is accounted for in the model.

\section{Main Contributions}
The sheaf-based coordination framework offers several notable contributions and advantages compared to classical graph-based methods:

\begin{enumerate}
    \item \textbf{Unified Modeling of Heterogeneous Constraints:} It provides a single formalism to model a wide range of multi-agent coordination problems. Classical approaches often treat consensus, formation, sensor fusion, etc., separately, each with its own custom algorithm. In contrast, by using cellular sheaves, all these problems become instances of finding consistent sections under certain constraints. This unification illuminates deep connections between problems (e.g., consensus and formation are both about aligning sections of a sheaf) and allows one to transfer insights and techniques across domains.
    \item \textbf{Rich Expressiveness:} Sheaves allow the modeling of vector-valued and structured data on networks. This means we can naturally handle \emph{asymmetric or directed interactions} (via non-symmetric restriction maps), multi-dimensional relationships (each edge can enforce a vector equation), and constraints that are not simply pairwise equality (they can be linear transformations or even nonlinear relations when coupled with potentials). This goes beyond what a weighted graph Laplacian can represent.
    \item \textbf{Topology-Aware Consistency:} By leveraging the machinery of cohomology (through $\delta$ and $H^0, H^1$), we gain access to tools like Hodge theory to reason about when coordination constraints have solutions and how “far” a given state is from any feasible solution (measured by cohomology classes). This is particularly useful in detecting and quantifying \emph{inconsistent loops or global obstructions} in the network constraints, something graph-theoretic approaches struggle to formalize. For example, if agents have cyclic dependencies that cannot all be satisfied (an unsolvable constraint loop), $H^1(G;\mathcal{F})\neq 0$ will reveal that.
    \item \textbf{Distributed Solver with Theoretical Guarantees:} The framework doesn’t just pose the problem but also provides a systematic way to solve it via ADMM and nonlinear sheaf Laplacian flows. Under convexity assumptions, we have convergence guarantees to global optima\footnote{In particular, ADMM for convex problems is known to converge under mild conditions. See Boyd et al., “Distributed Optimization and Statistical Learning via the Alternating Direction Method of Multipliers,” \textit{Foundations and Trends in Machine Learning} 3, no. 1 (2011): 1–122, for ADMM convergence theory. Hanks et al. extend this to the sheaf setting and prove convergence and optimality for the homological program solution.}. Even for certain nonconvex problems, the approach (as a heuristics) tends to perform well by exploiting problem structure. The distributed algorithm (Algorithm 1 in \cite{HanksEtAl}) clearly delineates the local computations and neighbor communications, facilitating implementation.
    \item \textbf{Link to Neural Network Perspective:} As we elaborate in Section 9, unrolling the iterations of the distributed solver yields a neural network-like architecture. This connection means we can potentially tune or learn parts of the coordination protocol (e.g., the edge potential parameters) using data, and conversely apply insights from deep learning (like initialization or normalization strategies) to improve coordination algorithms. It sets the stage for a synergy between rigorous optimization-based control and learning-based adaptation.
\end{enumerate}

These contributions make the sheaf-based approach a promising paradigm for complex multi-agent coordination going forward. It generalizes classical consensus and graph Laplacian methods (recovering them as special cases), and offers a pathway to integrate additional structure and objectives without losing the ability to analyze or solve the problem.

\section{General Multi-Agent Coordination Homological Program}
For clarity, we formally summarize the formulation of a general multi-agent coordination problem as a homological program:
\begin{equation*}
P^* \;=\; \underset{x \in C^0(G;\mathcal{F})}{\arg\min} \;\; \sum_{i\in V} f_i(x_i) \;\; \text{s.t.} \;\; \delta_{\mathcal{F}} x = 0~,
\end{equation*}
where 
\begin{itemize}
    \item $G=(V,E)$ is the agent communication graph;
    \item $\mathcal{F}$ is a cellular sheaf on $G$ capturing the data dimensions and consistency maps;
    \item $f_i(x_i)$ is the local cost for agent $i$ (convex, proper, lower semicontinuous);
    \item The constraint $\delta_{\mathcal{F}} x = 0$ means $x$ is a global section of the sheaf (all inter-agent consistency requirements are satisfied).
\end{itemize}
If needed, the constraint can be relaxed or penalized by introducing edge potentials $U_e$ as described earlier, yielding the unconstrained form:
\[
P = \underset{x \in C^0(G;\mathcal{F})}{\arg\min} \;\; \sum_{i} f_i(x_i) + \sum_{e} U_e\!\big((\delta_{\mathcal{F}} x)_e\big)~,
\]
which for sufficiently large penalty (or in the limit of hard constraints) recovers the constrained problem. This $P$ is the general multi-agent coordination homological program.

In practice, to use this framework one goes through the modeling steps exemplified in Section 5: identify the appropriate sheaf (data structure) and costs from the problem requirements, then set up $P$. Once $P$ is set up, we turn to solving it in a distributed manner.

\section{Solving Homological Programs with ADMM}\label{sec:admm}
We now outline the solution procedure using ADMM, following the approach of Hanks \emph{et al.} \cite{HanksEtAl}. The key idea is to split the problem across agents and edges so that each can be handled locally, and a simple iteration enforces agreement.

\subsection{Assumptions for Distributed Optimization}\label{sec:assumptions}
To ensure convergence of the ADMM-based solver, we assume:
\begin{enumerate}
    \item \textbf{Convexity:} Each $f_i: \mathcal{F}(i)\to \mathbb{R}\cup\{\infty\}$ is convex (and closed, proper), and each $U_e: \mathcal{F}(e)\to \mathbb{R}\cup\{\infty\}$ is convex. This makes the overall objective convex.
    \item \textbf{Strong Convexity of Potentials:} Each edge potential $U_e$ is strongly convex and differentiable, with a unique minimizer $b_e$ (interpreted as the “target” consistent value on that edge. Moreover, assume there is no fundamental inconsistency: the set of global sections achieving $b_e$ on each edge is non-empty (i.e. $b = \{b_e\}$ lies in $\mathrm{im}\,\delta_{\mathcal{F}}$).
    \item \textbf{Feasibility or Slater’s condition:} There exists at least one assignment $x$ such that $f_i(x_i)<\infty$ for all $i$ and $\delta x = 0$ (or in the penalized version, the constraints are soft so feasibility is trivial). Basically, the problem is well-posed.
\end{enumerate}
These conditions allow us to apply ADMM and guarantee it converges to the unique optimal solution of $P^*$.

\subsection{Consensus Form Reformulation}\label{sec:consensus-form}
We introduce local copies for the coupled variables to get an ADMM-friendly form. Let $z = \{z_e\}_{e\in E}$ be auxiliary variables intended to equal $\delta_{\mathcal{F}} x$. The constraint becomes $ \delta_{\mathcal{F}} x - z = 0$. We incorporate hard constraints with an indicator function $\chi_C(z)$ for $C:=\{z\,|\,z=0\}$, or if using $U_e$, it's already in the objective. For conceptual clarity, assume the hard form:
\[
\min_{x,z} \sum_{i} f_i(x_i) + \chi_C(z) \quad \text{s.t.}\quad \delta_{\mathcal{F}}x - z = 0~.
\]
This is equivalent to our original problem. The augmented Lagrangian (with scaled dual variable $y = \{y_e\}_{e\in E}$ where $y_e \in \mathcal{F}(e)$ for edge constraints) is:
\[
\mathcal{L}_\rho(x,z,y) = \sum_{i} f_i(x_i) + \chi_C(z) + \frac{\rho}{2}\|\,\delta_{\mathcal{F}} x - z + y\,\|^2_{2}~,
\] 
where $y$ is essentially the dual variable divided by $\rho$ (this is the typical "scaled" form of ADMM \cite{BoydADMM}). The ADMM updates are then:
\begin{align}
x^{k+1} &= \arg\min_{x}\; \sum_i f_i(x_i) + \frac{\rho}{2}\|\delta_{\mathcal{F}} x - z^k + y^k\|^2~, \label{update:x}\\
z^{k+1} &= \arg\min_{z}\; \chi_C(z) + \frac{\rho}{2}\|\delta_{\mathcal{F}} x^{k+1} - z + y^k\|^2~,\label{update:z}\\
y^{k+1} &= y^k + \delta_{\mathcal{F}} x^{k+1} - z^{k+1}~. \label{update:y}
\end{align}
Because of separability, \eqref{update:x} breaks into independent problems for each $x_i$ (each $f_i$ plus quadratic terms involving $x_i$ and its neighbors' data), and \eqref{update:z} is independent for each edge (since $\chi_C(z)$ decomposes over edges, enforcing each $z_e=0$). In fact, \eqref{update:z} yields simply 
\[
z^{k+1} = \Pi_C(\delta_{\mathcal{F}} x^{k+1} + y^k)~,
\] 
i.e. $z^{k+1} = \delta_{\mathcal{F}} x^{k+1} + y^k$ projected onto $C$ (which is zero, so $z^{k+1} = 0$ and effectively $\delta_{\mathcal{F}} x^{k+1} + y^k$ gets added to $y$). This simplifies the updates:
\begin{align*}
x^{k+1}_i &= \arg\min_{x_i} \; f_i(x_i) + \frac{\rho}{2}\sum_{e\ni i}\|\mathcal{F}_{i\to e}(x_i) - (\mathcal{F}_{j\to e}(x^k_j) - y^k_e)\|^2, \\
y^{k+1}_e &= y^k_e + (\mathcal{F}_{i\to e}(x^{k+1}_i) - \mathcal{F}_{j\to e}(x^{k+1}_j) )~,
\end{align*}
where $e=\{i,j\}$ and an orientation is assumed. We dropped $z$ entirely (since the projection simply enforces $z_e=0$ each time).

If we instead had the soft constraint version with $U_e$ in objective, the $z$-update would minimize $U_e(z_e) + \frac{\rho}{2}\|z_e - (\mathcal{F}_{i\to e}(x^{k+1}_i) - \mathcal{F}_{j\to e}(x^{k+1}_j) + y_e^k)\|^2$ for each edge $e$. That gives the optimality condition:
\[
\partial U_e(z^{k+1}_e) + \rho\big(z^{k+1}_e - [\mathcal{F}_{i\to e}(x^{k+1}_i)-\mathcal{F}_{j\to e}(x^{k+1}_j) + y^k_e]\big) \ni 0~.
\]
One can interpret this as computing a proximal operator of $U_e$. Solving this exactly may require a centralized step or iterative method, which is where the idea of using a \emph{diffusion process} to compute it arises (next subsection).

\subsection{Nonlinear Sheaf Diffusion for the Edge Update}\label{sec:diffusion-update}
To compute $z^{k+1}$ (or equivalently to enforce $\delta x^{k+1}$ satisfies the constraint) in a distributed way, one can simulate the continuous-time nonlinear heat flow:
\[
\dot{w}_e(t) = -\nabla U_e(w_e) - \rho\big(w_e - [\mathcal{F}_{i\to e}(x^{k+1}_i)-\mathcal{F}_{j\to e}(x^{k+1}_j) + y^k_e]\big)~,
\] 
for each edge $e$, which converges to $w_e^* = z^{k+1}_e$. This is essentially a gradient descent in $w$ for the objective in \eqref{update:z}. As Hanks et al. show, running this to convergence is equivalent to solving a \textbf{nonlinear Laplacian equation} on the network. In particular, assembling all these $w_e$ equations corresponds to the nonlinear sheaf Laplacian dynamics:
\[
\dot{x} = -L_{\nabla U}^{\mathcal{F}} x~,
\] 
which was given in \eqref{eq:nonlinear-laplacian-dyn} earlier. Intuitively, this diffusion allows neighboring agents to repeatedly adjust their tentative values to reach a consensus that minimizes the edge potentials. In implementation, one may not run it to full convergence every iteration; instead, interleaving one or a few diffusion steps between $x$-updates can be effective and still converge (this becomes a synchronous relaxation method).

The result is that the $z$-update (or equivalently enforcing consistency) can be done with purely local neighbor-to-neighbor interactions, without any centralized coordination. Theorem 8.1 in \cite{HanksEtAl} ensures that if the desired consistent edge values $b_e$ are attainable, this diffusion will find the projection of the current state onto the constraint manifold (global consistency.

\subsection{Summary of the Distributed Algorithm}
Putting the pieces together, a high-level description of the distributed ADMM algorithm for the homological program is:

\medskip
\noindent \textbf{Algorithm: Distributed Sheaf-ADMM}
\begin{enumerate}
    \item Initialize $x_i(0)$ for all agents (e.g., to some initial measurement or guess), and dual variables $y_e(0)=0$ for all edges.
    \item For $k = 0,1,2,\ldots$ (iterate until convergence):
    \begin{enumerate}
        \item \textbf{Local primal update:} Each agent $i$ receives the current dual variables $y_e^k$ from each incident edge $e=\{i,j\}$ (or equivalently receives $y_e$ from a designated orientation). Then agent $i$ updates its state:
        \[
        x_i^{k+1} := \arg\min_{x_i} f_i(x_i) + \frac{\rho}{2}\sum_{j \in N_i}\| \mathcal{F}_{i\to (i,j)}(x_i) - \mathcal{F}_{j\to (i,j)}(x^k_j) + y^k_{(i,j)} \|^2~,
        \] 
        which it can solve on its own (often a simple quadratic minimization or proximal operation).
        \item \textbf{Sheaf diffusion (consistency enforcement):} For each edge $e=\{i,j\}$, the two agents $i,j$ engage in a local exchange to adjust their states toward satisfying the constraint. This can be done by iterating:
        \[
        x_i \leftarrow x_i - \eta\, \mathcal{F}_{i\to e}^T \nabla U_e(\mathcal{F}_{i\to e}(x_i) - \mathcal{F}_{j\to e}(x_j)),
        \] 
        \[
        x_j \leftarrow x_j - \eta\, \mathcal{F}_{j\to e}^T \nabla U_e(\mathcal{F}_{i\to e}(x_i) - \mathcal{F}_{j\to e}(x_j)),
        \]
        for a few small steps $\eta>0$. In practice, $i$ and $j$ might repeatedly apply these updates using the latest information from each other until a certain tolerance or number of micro-iterations is reached. (This is a discrete implementation of the nonlinear heat equation on the edge.)
        \item \textbf{Dual update:} Each edge $e=\{i,j\}$ updates its dual variable:
        \[
        y_e^{k+1} := y_e^k + \mathcal{F}_{i\to e}(x_i^{k+1}) - \mathcal{F}_{j\to e}(x_j^{k+1})~,
        \] 
        which is the discrepancy after the $x$-update (the diffusion step is trying to minimize this, so ideally it is small). This $y_e$ is then sent to agents $i$ and $j$ to be used in the next iteration.
    \end{enumerate}
\end{enumerate}
\medskip

Each iteration involves agents computing updates using only neighbor information, and dual variables (which can be seen as “message variables”) being exchanged between neighbors. Convergence is determined, for instance, by the norm of primal residuals $\|\delta x^k\|$ and dual residuals $\|x^k - x^{k-1}\|$ falling below a threshold, as per standard ADMM criteria \cite{BoydADMM}.

\subsection{Convergence and Optimality}
Under the assumptions stated in Section \ref{sec:assumptions}, the algorithm converges to the global optimum $x^*$ of the convex program $P^*$. In the hard-constraint formulation, this means that eventually $\delta_{\mathcal{F}} x^* = 0$ (all constraints satisfied) and $x^*$ minimizes $\sum_i f_i(x_i)$ subject to that. In the soft formulation, it means $x^*$ minimizes the augmented objective with $U_e$ penalties (which often implies a good approximate satisfaction of constraints, or exact if penalties were enforced as barriers).

Hanks \emph{et al.} also discuss relaxing the strong convexity assumption: if $U_e$ are merely convex (not strongly), the solution may not be unique on the constraint manifold (there could be a family of equally optimal consistent states). In that case, ADMM can still converge, but the dual variables $y_e$ might not converge to a unique value. One remedy is adding a tiny regularization (strong convexification) to each $U_e$ to pick one solution. Alternatively, one can show that even without strong convexity, the primal variables $x$ still converge to \emph{a} solution (though possibly not uniquely determined by initial conditions if the problem has symmetry). This is a minor technical point — in most practical cases, constraints like consensus or formation have unique solutions except for trivial symmetries (like a common offset or rotation), which do not impede the functioning of the system.

By combining the power of convex optimization with sheaf-theoretic consistency constraints, we obtain a distributed algorithm that generalizes many known coordination protocols (for example, it reduces to the usual consensus algorithm when $\mathcal{F}$ is trivial and $f_i$ are quadratic). We next turn to interpreting these procedures through the lens of neural networks, which offers additional insight and potential for extension.

\section{From Sheaves and Laplacians to Neural Networks}
It is insightful to interpret the iterative distributed algorithm as a form of neural network that \emph{learns or converges} to a solution. Recently, researchers have introduced \textbf{Sheaf Neural Networks (SNNs)} as a generalization of Graph Neural Networks (GNNs) that operate on cellular sheaves rather than on plain graph. In this section, we draw parallels between our framework and SNNs, illuminating how multi-agent coordination can be seen as (or embedded into) a neural network architecture.

\subsection{Sheaf Neural Networks: Definition and Structure}
A Sheaf Neural Network layer, informally, takes an input assignment of feature vectors to the nodes of a graph (a 0-cochain with values in some feature sheaf) and produces an output assignment (a transformed 0-cochain) by a combination of:
\begin{itemize}
    \item A linear transformation on each node’s feature (analogous to a weight matrix in a neural layer, acting independently on each node).
    \item A propagation of information along edges that respects the sheaf structure. Typically, this involves applying restriction maps to node features, exchanging information across the edge, possibly applying some edge-specific weighting or transformation, and then aggregating back into node updates.
    \item A nonlinear activation function applied to the node features.
\end{itemize}
For example, an SNN as defined in \cite{Roddenberry2021SheafNN} uses the sheaf Laplacian in the following way: one can form a “sheaf convolution” operator $H = I - \epsilon L_{\mathcal{F}}$ for small $\epsilon$, which is analogous to the graph convolution $I - \epsilon L_G$ in GNNs. Stacking such operations with intermediate nonlinearities yields an SNN. 

In our context, consider the process of one ADMM iteration as a mapping $(x^k, y^k) \mapsto (x^{k+1}, y^{k+1})$. This mapping is reminiscent of a two-layer update:
\begin{itemize}
    \item The $x$-update can be seen as each node taking its current state $x_i^k$ and neighbor messages $y^k$ (which carry a summary of neighbor states from the previous iteration) and producing a new state $x_i^{k+1}$. If $f_i$ is quadratic, this is an affine combination of $x_i^k$ and neighbors’ contributions, possibly with a nonlinearity if $f_i$ had one (like a clipping if domain constrained). This is analogous to a neuron update: sum weighted inputs (neighbors’ states) and apply a (proximal) nonlinearity.
    \item The dual update can be seen as another layer where each edge (which could be regarded as a separate computational unit) receives the two endpoint states $x_i^{k+1}, x_j^{k+1}$ and outputs an updated message $y_e^{k+1}$. The formula $y_e^{k+1} = y_e^k + (\mathcal{F}_{i\to e}(x_i^{k+1}) - \mathcal{F}_{j\to e}(x_j^{k+1}))$ is essentially an accumulation of inconsistency (like an error signal). One could interpret $y_e$ as analogous to a hidden layer that stores “memory” about the edge consistency (like a recurrent connection).
\end{itemize}
If we unroll these iterations for a fixed number of steps $T$, we have a feedforward network of depth $T$, where the parameters of the network are determined by $f_i$, $U_e$, $\mathcal{F}_{i\to e}$, and $\rho$. Interestingly, these parameters are not learned in the coordination setting; they are fixed by the problem. But one could envision learning some of them (for instance, if $f_i$ had unknown weights that could be tuned to improve performance or adapt to environment).

Thus, the sheaf-based coordination algorithm itself \emph{is a kind of neural network} — specifically, a monotone operator network aimed at solving an optimization. This viewpoint is related to the concept of \textbf{algorithm unrolling} in machine learning, where one takes an iterative algorithm and treats a finite number of its iterations as a neural network layer sequence, possibly fine-tuning it using data.

\subsection{Node State Updates as Neuron Activations}
Examining \eqref{update:x}, the node update step for agent $i$:
\[
x_i^{k+1} = \mathrm{prox}_{f_i/\rho}\Big( \; x_i^k - \sum_{j\in N_i} \mathcal{F}_{i\to e}^T(y^k_e + \mathcal{F}_{i\to e}(x_i^k) - \mathcal{F}_{j\to e}(x_j^k)) \;\Big)~,
\] 
where $\mathrm{prox}$ denotes the proximal operator (which is identity minus gradient of $f_i$ if $f_i$ is differentiable). This equation closely resembles the update of a neuron that receives input $(x_i^k$ and neighbor terms$)$, performs a linear combination, and then applies a nonlinear function (the prox, which could be e.g. a truncation if $f_i$ encodes bounds, or a shrinkage if $f_i$ is an $\ell_1$ penalty, etc.). In the simplest case $f_i(x) = \frac{1}{2}\|x\|^2$, $\mathrm{prox}$ is just identity minus a linear term, and the update becomes 
$x_i^{k+1} = x_i^k - \alpha \sum_{j\in N_i}\mathcal{F}_{i\to e}^T(\mathcal{F}_{i\to e}(x_i^k) - \mathcal{F}_{j\to e}(x_j^k) + y^k_e)$, 
which is an affine function of the inputs (no nonlinearity). If $f_i$ enforces constraints like $x_i \ge 0$ (nonnegativity of some variables), then the prox includes a ReLU-like clip at 0. So the $x$-update layer can incorporate activation functions naturally depending on the form of $f_i$.

\subsection{Sheaf Diffusion as Message-Passing Layer}
The diffusion step (Section \ref{sec:diffusion-update}) can be viewed as a message-passing layer. In graph neural networks, a typical layer update is:
\[
h_i^{new} = \sigma\Big( W \cdot h_i^{old} + \sum_{j\in N_i} W'\cdot h_j^{old} \Big)
\] 
for some weight matrices $W, W'$ and nonlinearity $\sigma$. In our case, during diffusion each agent’s state is adjusted by something coming from each neighbor: $-\eta \mathcal{F}_{i\to e}^T \nabla U_e(\mathcal{F}_{i\to e}(x_i) - \mathcal{F}_{j\to e}(x_j))$. This is a message from neighbor $j$ to $i$, which depends on $x_i$, $x_j$, and the function $\nabla U_e$. In an SNN, one could have a learnable function or attention mechanism on each edge; here $\nabla U_e$ plays that role, determining how much influence the neighbor’s difference has. For example, if we set $U_e(y) = \frac{1}{2}\|y\|^2$, then $\nabla U_e(y) = y$ and the message is just $-(\mathcal{F}_{i\to e}^T \mathcal{F}_{i\to e}(x_i) - \mathcal{F}_{i\to e}^T \mathcal{F}_{j\to e}(x_j))$. If $U_e$ is something like a saturated loss that grows slowly after a point, $\nabla U_e$ might reduce large differences’ effect (similar to clipping gradients, which is analogous to robust message passing).

The diffusion is essentially a message-passing layer enforcing consistency, and it is naturally incorporated in the ADMM iteration. In a purely feed-forward SNN, one might include a similar diffusion or use the sheaf Laplacian in each layer to mix features.

\subsection{Role of Projection (Consensus) and Dual Variables}
The Euclidean projection $\Pi_C$ in the ADMM algorithm (which we implemented via diffusion) has an analogue in neural networks as well: it’s like a normalization or constraint satisfaction step between layers. Some neural network architectures include normalization layers (like batch norm or layer norm) to enforce certain conditions (zero mean, unit variance) at each layer. Here, $\Pi_C$ ensures that after each full iteration, the “intermediate” $z$ is in the feasible set of constraints. We can think of it as a layer that enforces the sheaf constraint approximately at every step. If one unrolled infinitely many layers, the final output would be in the constraint set exactly. With a finite number, it’s approximately satisfying constraints (much like how a sufficiently deep GNN can approximate broad functions, but a shallow one might not exactly solve a complex constraint, though it gets closer with each layer).

The dual variables $y_e$ can be seen as carrying the memory of past inconsistencies, which improves convergence (like momentum in optimization). In a network sense, $y_e$ are additional hidden features on edges accumulating information across layers. One might imagine an architecture where edges have hidden units that interact with node units. Indeed, some recent GNN models introduce edge features that evolve alongside node features. Here, $y_e$ is exactly that: an edge feature updated as $y_e^{k+1} = y_e^k +$ (neighbor difference). This resembles a recurrent connection (like an LSTM gating, albeit simpler linear accumulation) ensuring that if a constraint hasn’t been satisfied yet, the “pressure” to satisfy it builds up (encoded in $y_e$).

\subsection{Optimization View vs. Learning View}
A striking difference between our ADMM-unrolled network and a typical learned neural network is that our network’s “weights” and non-linearities are derived from an optimization problem, not trained from data. We are effectively doing \emph{predictive coding} where the “prediction” is the consistent optimal state and the “error” is driven to zero via iterative feedback (dual updates). On the flip side, if we had multiple coordination tasks or unknown parameters in $f_i, U_e$, we could imagine learning those parameters by gradient descent through this network. For instance, if $U_e(y)$ had a parameter that we want to tune so that the system converges faster, we could treat convergence speed or error as a loss and backpropagate through the unrolled iterations to adjust that parameter.

This cross-pollination of ideas suggests a future direction where control strategies for multi-agent systems are designed with neural network principles (modularity, trainable components) and conversely, neural networks for graph-structured data incorporate principled constraints via sheaf Laplacians (ensuring learned representations respect certain invariances or physical laws).

The correspondence is:
\begin{center}
\begin{tabular}{|c|c|}
\hline 
\textbf{Coordination Algorithm} & \textbf{Neural Network Analogy} \\
\hline
Agent states $x_i^k$ & Neuron activations/feature vectors \\
Dual variables $y_e^k$ & Edge hidden units / error signals \\
Local objective $f_i$ & Neuron activation function (via prox) \\
Edge potential $U_e$ & Message function / attention along edge \\
Iteration steps & Network layers (unrolled in time) \\
Converged $x^*$ & Network output after many layers \\
\hline
\end{tabular}
\end{center}

Understanding these analogies can inspire hybrid methods: one could use a few iterations of the theoretical ADMM (ensuring feasibility and decent optimum) and then use a neural network to fine-tune to real-world data or uncertainties, or use a neural network to warm-start ADMM. Or simply leverage hardware and software from deep learning to implement distributed solvers more efficiently.

\section{Advantages of Sheaf Neural Networks for Learning and Coordination}
Considering the above connections, we list some specific advantages that using sheaf neural network concepts (and, by extension, our sheaf coordination framework) brings:

\subsection{Modeling Asymmetric and Heterogeneous Interactions}
Standard GNNs often assume homogeneous scalar weights on edges (or at best a fixed vector of weights for a given edge type). Sheaf Neural Networks, however, inherently handle \textbf{asymmetric interactions}: since each edge has distinct restriction maps for each endpoint, the influence of node $i$ on node $j$ can differ from $j$ on $i$. This is crucial for capturing \textit{heterophily} in networks (cases where connected nodes may have very different states rather than similar). Empirical studies have shown SNNs excel in graphs with directed edges or where the relationship is not “mirror-symmetric". In multi-agent coordination, this means we can rigorously model situations like one agent’s output being another’s input (a directed dependence) or interactions that are not identical in both directions (like one agent might enforce a constraint on another but not vice versa). The use of sheaves thus broadens the class of systems and communication patterns that can be handled seamlessly.

\subsection{Nonlinear Diffusion and Improved Expressivity}
By incorporating nonlinear potentials $U_e$, the message-passing in the network becomes nonlinear in the differences between agents’ states. This \textbf{nonlinear diffusion} mechanism allows the network to implement more complex behaviors than a linear Laplacian does. For example, it can learn to ignore small differences (below some threshold) and only act strongly on large disagreements, or vice versa, akin to a robust or adaptive diffusion. From a learning perspective, this gives SNNs a richer function class to approximate. Recent work titled “Sheaf Diffusion Goes Nonlinear” demonstrates that introducing nonlinearity in the sheaf propagation can significantly enhance performance of GNNs on certain tasks (e.g., learning functions on graphs that normal GNNs struggle with due to oversmoothing or heterophily. In coordination, a nonlinear diffusion means the system can handle constraints that only need to kick in beyond certain thresholds (like collision avoidance might be negligible when far apart but strongly repulsive when too close, which is naturally modeled by a nonlinear potential).

\subsection{Decentralized Learning and Adaptation}
A sheaf neural network can be executed in a decentralized manner: each layer’s computation uses only local neighbor information. This means, in principle, an SNN can run on the multi-agent system itself. Agents could run a neural network that was trained (possibly offline, or even online in a continual learning way) to achieve coordination under certain conditions (including learned behaviors or heuristic improvements). This is different from standard central training of controllers. In essence, \textbf{decentralized learning} is enabled by the locality of SNN layers. Moreover, one could personalize the behavior for each agent by having agent-specific functions $f_i$ or features, which the framework naturally accommodates (each node can have its own parameters in the neural net or its own part of the objective in control terms). 

An exciting possibility is to use reinforcement learning or gradient-based learning to adjust the local objective functions $f_i$ or potentials $U_e$ based on performance feedback, all while maintaining the sheaf constraint structure. Because the sheaf formalism ensures consistency, any learned policy inherently respects the critical constraints (no need to relearn physical laws or communication protocols – they are “baked in” to the architecture). This makes safe learning more feasible: the network will not easily produce a grossly inconsistent action because the architecture itself penalizes that heavily.

The marriage of sheaf-based optimization with neural network methodology offers:
\begin{itemize}
    \item \emph{Greater modeling power} (through asymmetric, high-dimensional relations).
    \item \emph{Enhanced algorithmic performance} (through nonlinear message functions and the ability to incorporate learned elements).
    \item \emph{Maintained decentralization and interpretability}, since the core is still an optimization with a known objective, and convergence is analyzable.
\end{itemize}

This is a rapidly evolving area of research, and initial studies are promising in showing that SNNs can outperform classical GNNs on certain graph prediction tasks by better handling structured relation. Likewise, in multi-agent experiments, one can expect that incorporating sheaf-based design leads to more robust and flexible coordination behaviors than purely linear approaches.

\section{Conclusion and Future Directions}
In this paper, we presented an expository overview of cellular sheaves and their associated Laplacians, and showed how they provide a powerful framework for multi-agent coordination. We started from basic definitions (sheaves on graphs, sections, and Laplacians) and built up to advanced applications (nonlinear homological programs and distributed ADMM solvers). Throughout, we highlighted the intuition behind the mathematics: sheaves generalize graphs by attaching vector spaces and constraints, and sheaf Laplacians generalize graph Laplacians to enforce these constraints in a dynamical or optimization context.

We demonstrated how classical coordination tasks like consensus and formation control can be formulated in this framework, and how a unified distributed algorithm can solve all of them by exploiting convexity and locality. A significant portion of our discussion was devoted to drawing connections between this algorithm and neural network concepts, which serve to deepen understanding and pave the way for new hybrids of learning and control.

Several key takeaways are:
\begin{itemize}
    \item Cellular sheaves allow encoding of complex agent relationships, overcoming limitations of simple graph models (especially for heterogeneous systems).
    \item The nonlinear sheaf Laplacian provides a principled way to design distributed controllers that achieve global objectives while respecting local constraints, bridging spectral graph theory with nonlinear control.
    \item ADMM on sheaf-structured problems yields a message-passing algorithm that is naturally decentralized and comes with convergence guarantees under convexity assumptions.
    \item Viewing the iterative solver as a neural network unveils opportunities to integrate learning, enabling systems that can potentially adapt to unknown conditions by tuning their local cost functions or interaction potentials.
\end{itemize}

\textbf{Future Directions:} This area is ripe with avenues for further research. A few prominent ones include:
\begin{enumerate}
    \item \emph{Extensions to Time-Varying and Stochastic Environments:} Real multi-agent systems operate in changing conditions. One can extend the sheaf framework to time-evolving sheaves (where the graph or restriction maps change over time) or uncertain data (where measurements are noisy, suggesting a need for stochastic potentials or Bayesian interpretations).
    \item \emph{Higher-Dimensional Sheaves:} We mainly discussed sheaves on graphs (1-dimensional cell complexes). The theory extends to sheaves on higher-dimensional cell complexes (hypergraphs or simplicial complexes). This could model multi-agent interactions involving more than two agents at once (e.g., a constraint among a trio of agents). The concept of a sheaf Laplacian also extends (leading to higher Hodge Laplacians). Distributed coordination on such complexes (perhaps using tools from higher-order network theory) is an exciting frontier.
    \item \emph{Learning Restriction Maps from Data:} In some cases, one might not know the exact linear relationship between agent states that should hold; one could try to learn the linear maps $\mathcal{F}_{i\to e}$ from demonstrations of coordinated behavior (as was explored in work on learning sheaf Laplacians from signals on graphs by Hansen \& Ghrist). This would reverse-engineer the constraints that a group of agents is implicitly following.
    \item \emph{Hardware Implementation and Scalability:} As systems grow large (e.g., swarms of hundreds of drones), implementing the sheaf computations efficiently is important. Graph GPU libraries and neural network hardware could be leveraged given the similarities we've noted. Investigating the scalability and real-time performance of these algorithms in large networks is crucial for practical adoption.
    \item \emph{Integration with Control Theory:} Finally, connecting this more with continuous control: for instance, using the sheaf Laplacian flow as a controller in continuous time, or analyzing stability in the presence of communication delays and quantization. The rich structure may allow more precise statements about robustness (since $\ker L_{\mathcal{F}}$ characterizes the steady-state agreement exactly, and one can possibly quantify how perturbations move the system off that manifold).
\end{enumerate}

We believe that the cellular sheaf perspective will become increasingly important as networks and systems become more complex and require integrated approaches from topology, algebra, optimization, and learning. By being both a comprehensive tutorial and a pointer to recent advances (such as those by Hanks et al. and Riess), we hope this paper serves as a stepping stone for researchers and practitioners to explore and apply sheaf-based methods in their own domains.

\end{document}